\documentclass[12pt]{amsart}
\usepackage{amsthm,amssymb,amsmath,enumerate,graphicx, tikz}
\usepackage{amscd}
\usetikzlibrary{decorations.markings}
\usepackage{setspace}
\usepackage{comment}

\theoremstyle{plain}
\newtheorem{theorem}{Theorem}[section]

\newtheorem{lemma}[theorem]{Lemma}
\newtheorem{observation}[theorem]{Observation}
\newtheorem{corollary}[theorem]{Corollary}
\newtheorem{claim}[theorem]{Claim}
\newtheorem{conjecture}[theorem]{Conjecture}

\theoremstyle{definition}
\newtheorem{definition}[theorem]{Definition}

\newtheorem{prob}[theorem]{Problem}

\theoremstyle{remark}
\newtheorem{example}[theorem]{Example}

\newtheorem{remark}[theorem]{Remark}

\newcommand{\cf}{\mathcal{F}}
\newcommand{\F}{\mathcal{F}}
\newcommand{\I}{\mathcal{I}}
\newcommand{\J}{\mathcal{J}}

\newcommand{\N}{\mathbb{N}}
\newcommand{\C}{\mathcal{C}}

\newcommand{\cb}{\mathcal B}

\newcommand{\cm}{\mathcal M}

\newcommand{\ci}{\mathcal I}
\newcommand{\cc}{\mathcal C}
\newcommand{\cd}{\mathcal D}
\newcommand{\cg}{\mathcal G}
\newcommand{\cx}{\mathcal X}
\newcommand{\cu}{\mathcal U}
\newcommand{\cn}{\mathcal N}
\newcommand{\ct}{\mathcal T}

\newcommand{\ithree}{I_{\geq 3}}

\newcommand{\Cn}{\overrightarrow{C_n}}
\newcommand{\Ck}{\overrightarrow{C_k}}
\newcommand{\Cs}{\overrightarrow{C_s}}

\newcommand{\cxk}{{\mathcal{X}(k)}}

\newcommand{\f}{f'}
\newcommand{\fr}{f^\text{r}}
\title{Rainbow independent sets in certain classes of graphs}
\author{Ron Aharoni}\thanks{The research of the first  author was
supported by  BSF grant no.
2006099, by an ISF grant and by the Discount Bank
Chair at the Technion.}
\address{Department of Mathematics,
Technion\\
Haifa, Israel} \email{raharoni@gmail.com}

\author{Joseph Briggs}
\address{Department of Mathematics,
Technion\\
Haifa, Israel}\email{briggs@campus.technion.ac.il}

\author{Jinha Kim}
\address{Department of Mathematics,
Technion\\
Haifa, Israel}\email{jinhakim@campus.technion.ac.il}

\author{Minki Kim}
\address{Department of Mathematics,
Technion\\
Haifa, Israel}\email{kimminki@campus.technion.ac.il}
\begin{document}

\maketitle
\begin{abstract}
    For a given class $\cc$ of graphs and given integers $m \le n$, let $f_\cc(n,m)$ be the minimal number $k$ such that every $k$ independent $n$-sets in any graph belonging to $\cc$ have a (possibly partial) rainbow independent $m$-set.    
    Motivated by known results on the finiteness and actual value of $f_\cc(n,m)$ when $\cc$ is the class of line graphs of graphs, we study this function for various other classes. 
\end{abstract}
\section{Introduction}
The protagonists of this paper are  {\em rainbow sets}. 
\begin{definition}
Let $\cf=(F_1, \ldots ,F_m)$ be a collection of (not necessarily distinct) sets. A (partial) {\em rainbow set} for $\cf$ is 
the image of a partial choice function.
More formally - it is a
set of the form $R = \{x_{i_1}, x_{i_2},\ldots ,x_{i_k}\} $, where $1 \le i_1<i_2< \ldots <i_k\le m$, and $x_{i_j}\in F_{i_j} ~~(j \le k)$. Here it is assumed that $R$ is a set, namely that the elements $x_{i_j}$ are distinct.
% If $k=m$ (namely if all $F_j$'s are represented),  then $R$ is called a {\em full rainbow set}.
 \end{definition}

An {\em $n$-set} is a set of size $n$. A set of vertices in a graph is called {\em independent} if it does not contain an edge of the graph. The set of independent sets in a graph $G$ is denoted by $\ci(G)$, and
the set of independent $n$-sets is denoted by $\ci_n(G)$. 
The maximal size of an independent set in $G$ is denoted by $\alpha(G)$.

If $K,~ H$ are two graphs, we write $H<K$ if $K$ contains an induced copy of $H$. If $H \not < K$ we say that $K$ is {\em $H$-free}.

If $u,v$ are adjacent in a given graph, we write $u \sim v$.
By $N[v]$  we denote the set $\{v\} \cup \{u \mid u\sim v\}$, and by $N(v)$ the set $\{u \mid u \sim v\}$. 

\vskip\baselineskip
\begin{definition}
Let $H$ be a hypergraph. The {\em line graph}
of $H$, denoted by $L(H)$, has the edges of $H$ as vertices, and two vertices are adjacent if they intersect, as edges of $H$. \end{definition}

 For a graph $G$ and integers $m \le n$, let $f_G(n,m)$ be the minimal number $k$ such that every $k$ independent $n$-sets in $G$ have a partial rainbow independent $m$-set. For a class $\cc$ of  graphs, let $f_\cc(n,m)=\sup\{f_G(n,m) \mid G \in \cc\}$. This can be $\infty$. The aim of this paper is to establish bounds on the values of $f_\cc(n,m)$ for certain classes $\cc$.

In particular, 
we shall consider the following classes:

\begin{enumerate}[1.]
    
     \item $\cu$:  the class of all graphs.
    \item $\cb$:  the class of line graphs of bipartite graphs.
     \item $\cg$:  the class of line graphs of all graphs.
     \item $\cx(k)$: the class of $k$-colourable graphs.
     
     \item $\cd(k)$: the class of  graphs with degrees at most $k$.
     \item $\ct$:  the class of chordal graphs.
     \item $\cf(H)$: the class of $H$-free graphs, for a given graph $H$.
     \item $\cf(H_1, \dots, H_t) = \bigcap_{i\le t} \cf(H_i)$: the class of graphs that are $H_i$-free for all $i\le t$.
\end{enumerate}

Here is a small example, for practice.
\begin{example}
 For every $k$, let $G$ be the complete $k$-partite graph with all sides of size $n$, and let $F_i$ be its respective sides. Then there is no independent rainbow $2$-set, which shows that $f_\cu(n,2)=\infty$ for every $n$.
\end{example} 
Clearly, 
\begin{equation}\label{trivialbound}
   f_\cc(n,m) \ge m
\end{equation}
%(Leaving out  graphs with no independent $n$-set.)
(provided $\cc$ has at least one graph with an independent $n$-set).
\\

If $\cc \subseteq \cd$ and  $m' \le m \le n \le n'$ then:
\begin{equation}\label{monotonepartial}
   f_\cd(n,m) \ge f_\cc(n',m').
   % \; \text{ and } \;
   % f_{\cc}(n,m) \ge f_{\cc}(n-m',m-m').
\end{equation}

Here are some of the main results of the paper.

    \begingroup
    \def\thetheorem{\ref{t.forbidden_introduction}}
    
    Let $K_{r}^-$ denote the complete graph on $r$ vertices, with one edge deleted.

    \begin{theorem}
    % Let $n \ge 4$. Then $f_{\cf(H)}(n,n) < \infty$ if and only if  $ H \simeq K_r$ or $ K_{r}^-$ for some $r$.
    $f_{\cf(H)}(n,n) < \infty$ for every positive integer $n$ if and only if $H$ is either $K_r$ or $K_{r}^-$ for some $r$.
    \end{theorem}
    \addtocounter{theorem}{-1}
    \endgroup

	\begingroup
	\def\thetheorem{\ref{t.chordal rb}}
	\begin{theorem}
    If   $m \le n$ then $f_\ct(n,m)=m.$
    \end{theorem}
	\addtocounter{theorem}{-1}
	\endgroup

	\begingroup
	\def\thetheorem{\ref{t.colourable}}
	\begin{theorem}
    If  $m \le n$ then   $f_{\cx(k)}(n,m)=(m-1)k+1.$
    \end{theorem}
	\addtocounter{theorem}{-1}
	\endgroup

%The most voluminous section, 
Section \ref{bounded_degrees} is devoted to the class $\cd(k)$, for any integer $k$. 
For this class the values of $f(n,m)$ are only conjectured, and we shall prove only some special cases of the conjecture, as well as a weaker result.

\section{Rainbow matchings in graphs}
Part of the motivation for the study of the functions $f_\cc$ comes from the case  $\cc=\cg$. In this case the independent sets are matchings in graphs. 
A prototypical result  is a theorem of   Drisko \cite{drisko}. In a slightly  generalized form,  proved in \cite{AB},  it states that every $2n-1$ matchings of size $n$ in a bipartite graph have a partial rainbow matching of size $n$. Since an independent set in $L(H)$  is a matching in $H$, this can be stated as:

\begin{theorem}[Drisko]\label{drisko}
$f_\cb(n,n) \le 2n-1$.
\end{theorem}

In fact, equality holds, by the following example.

\begin{example} \label{driskosharp}
Take the two matchings of size $n$ in the cycle $C_{2n}$, each  repeated $n-1$ times. These are $2n-2$ independent $n$-sets in a line graph, having no rainbow independent $n$-set. 
\end{example}

In \cite{BGS} the following was proved:

\begin{theorem}
$f_\cb(n, n-k) \le \lfloor \frac{k+2}{k+1}n \rfloor -(k+1)$.
\end{theorem}

In particular, $f_\cb(n, n-1) \le \lfloor \frac{3}{2}n \rfloor -2$.

In \cite{ABCHS} a bound was proved also for line graphs of general graphs:

\begin{theorem}\label{gendrisko}
$f_\cg(n,n) \le 3n-2$.
\end{theorem}

Guided by  examples from \cite{BGS}, the following was conjectured there:

\begin{conjecture}\label{c.ABversion 1}
 $f_\cg(n,n) = 2n$ and for $n$ odd $f_\cg(n,n) =2n-1$. 
\end{conjecture}

In Remark \ref{goodreason} below we shall give  supporting evidence for this conjecture.

There is an example showing that for $n$ even $f_\cg(n,n) \ge 2n$. Since we shall use it below, we describe it explicitly. 

\begin{example}\cite{ABCHS}\label{neven}
Let $n=2k$. In the cycle $C_{2n}$ repeat each of the two perfect matchings $n-1$ times. To this add a perfect matching $N$ consisting solely of edges of even length (such $N$ can be shown to exist if and only if $n$ is even). Since an even length edge encloses an odd set, that cannot be matched within itself, the edges of $N$ cannot be used for a rainbow perfect matching. Thus we are back in the situation of Example \ref{driskosharp}, in which as we saw there is no perfect rainbow matching. 
\end{example}

In \cite{AHZ} the following fractional version was proved: 
\begin{theorem}
Let  $F_1, \dots ,F_{2n}$ be sets of edges in a  graph $G$. 
If $\nu^*(F_i) \ge n$ for each $i\le 2n$, then there exists a partial rainbow set $F$ of edges with $\nu^*(F) \ge n$. If $G$ is bipartite, then any  $2n-1$ sets $F_i$ with $\nu^*(F_i) \ge n$ have a partial rainbow set $F$ of edges with $\nu^*(F) \ge n$.
\end{theorem}

Since in bipartite graphs $\nu^*=\nu$, the second part of the theorem  is a re-formulation of Drisko's theorem. 

  In \cite{AB} the following was conjectured:

\begin{conjecture}\label{ABversion2}
$f_\cb(n,n-1) \leq n$.
\end{conjecture}

This means: every $n$ matchings of size $n$ in a bipartite graph have an $(n-1)$-rainbow matching. This is a generalization of a famous conjecture or Ryser \cite{ryserlatin} on transversals in Latin squares. 
In fact, this may well be true also for non-bipartite graphs: we do not know a counterexample to the stronger $f_\cg(n,n-1)=n$. 
Note the surprising jump from Conjecture \ref{c.ABversion 1}- raising $m$ from $n-1$ to $n$ almost doubles $f_\cg(n,m)$. The secret of this jump may be somewhat elucidated if the following is true: 

\begin{conjecture}
If $s < \frac{t}{2}$ then $f_{C_t}(s,s)=s$. 
\end{conjecture}

A stronger version of Conjecture \ref{ABversion2} 
is:

\begin{conjecture}\label{abstrong}
$f_\cb(n,n-1)=n-1$. 
\end{conjecture}

Namely, every $n-1$ matchings of size $n$ in a bipartite graph have a full rainbow matching.
%(Conjecture \ref{ABversion2} follows by adding the same ``dummy'' edge, disjoint from all other edges, to all matchings in the family).
As stated, this conjecture is false for  general graphs.  

\begin{example}\label{twok4}
Let $n$ be even. Take Example \ref{neven}, and multiply it by $2$. This means making another copy $V'$ of the vertex set, and adding to each matching its copy in $V'$.
This results in $2n-1$ matchings, each of size  $2n$. We claim that  there is no rainbow matching of size $2n-1$. If such existed, $n$ of its edges would be in one copy of Example \ref{neven}, contradicting the fact that this example does not possess a rainbow matching of size $n$.   
\end{example}

\begin{conjecture}\label{strangebuttrue}
If $G$ is the line graph of a graph, and it is not a graph as in Example \ref{twok4}, then $f_G(n,n-1)=n-1$.
\end{conjecture}

\begin{remark}\label{goodreason}
The way Example \ref{twok4}  is constructed shows that Conjecture \ref{strangebuttrue} implies Conjecture \ref{c.ABversion 1}: if the latter fails, then the construction provides a counterexample to the first.
\end{remark} 

In Section \ref{bounded_degrees} (see Remark \ref{rem.k=2revisited})
we shall see an explanation for  the mysterious jump from $n-1$ in the conjecture to $2n-1$ in Theorem \ref{drisko}. 
We shall meet there  an interesting special case of Conjecture \ref{strangebuttrue}:

\begin{conjecture}\label{k=2ab}
If $G$ is a graph of maximal degree $2$, then 
$f_G(n,n-1)=n-1$. 
\end{conjecture}

Since every graph $G$ of maximum degree 2 is the line graph $L(H)$ for some $H$ of maximum degree 2, this means the following: every $n-1$ matchings of size $n$ in a graph of maximal degree $2$  have a full rainbow matching. If the graph contains only one cycle, then the conjecture follows from Theorem   \ref{t.chordal rb} below, stating that in chordal graphs $f(n,n)=n$ (just remove one vertex from the cycle, making the graph an interval graph, and hence chordal).

Here are two more results strengthening 
Theorem \ref{drisko}.
One is a matroidal version, by Kotlar and Ziv:

\begin{theorem} \cite{kotlarziv}
If $\cm, \cn$ are matroids on the same ground set, then any $2n-1$ sets belonging to $\cm \cap \cn$ have a rainbow set of size $n$ belonging to $\cm \cap \cn$.
\end{theorem}
Theorem \ref{drisko} is the special case in which both matroids are partition matroids. 
Another stronger version appeared in \cite{AKZ}, where it was shown that not all matchings need to have size $n$:
%a strengthening of Theorem \ref{drisko} was proved:

\begin{theorem}\label{strongerdrisko}
If $\cf=(F_1, \ldots ,F_{2n-1})$ is a family of matchings in a bipartite graph, and $|F_i|\ge \min(i,n)$ for every $i \le 2n-1$, then  $\cf$ has a rainbow matching. 
\end{theorem}

%Some of the results will be given, in addition to purely combinatorial proofs, also topological proofs, based on a theorem of Meshulam and Kalai. These proofs are concentrated in Section \ref{top}.

\section{Graphs avoiding given induced subgraphs}

In this section we study the function $f_\cc$ for classes of the form $\cc=\cf(H_1, \ldots ,H_m)$.  

 \begin{observation}\label{montonicity2}
Let $H,K$ be graphs. If $H < K$,
%and $f_{\cf(H)}(n,n)=\infty$, then $f_{\cf(K)}(n,n)=\infty$. 
then $f_{\cf(H)}(n,n) \leq f_{\cf(K)}(n,n)$.
\end{observation}

The graph $K_{1,3}$ is called a ``claw''. Line graphs are claw-free, and though this does not 
characterize them, many properties of line graphs follow from mere claw-free-ness. This is not the case here.

\begin{theorem}\label{clawfree} Let $t \geq 1$.
\begin{enumerate}[(a)]
    \item $f_{\cf(K_{1,t+1})}(n,m) =m $ for $1 \leq m \leq \left\lceil \frac{n}{t} \right\rceil $.
    \item $f_{\cf(K_{1,t+1})}(n,m)=\infty $ for $m > \left\lceil \frac{n}{t} \right\rceil$.
\end{enumerate}
\end{theorem}

\begin{proof}
\begin{enumerate}[(a)]
    \item By \eqref{trivialbound} it suffices to prove $f_{\cf(K_{1,t+1})}(n,m) \leq m$. We apply induction on $m$. The base case $m=1$ requires just noting that any vertex in a single $n$-set is a rainbow 1-set. 
    
    For the inductive step, suppose $K_{1,t+1} \not < G$, and let  $I_1, \ldots ,I_m \in \ci_n(G)$. 
Pick a vertex $v \in I_m$.
For every $I_j$, $j < m$, either $v \in I_j$ or $v$ has at most $t$ neighbors in $I_j$.
%Let $N[v]$ denote the {\em closed neighbourhood} of $V$, consisting of $v$ and all its neighbors, and
Consider $I_j' := I_j \setminus N[v]$ for $j < m$.
$I_1', \dots, I_{m-1}' \in \ci(G)$, and have size at least $n-t$, and we note that $m -1 \leq \left\lceil\frac{n-t}{t}\right\rceil$ since $m \leq \left\lceil \frac{n}{t} \right\rceil$.
By the induction hypothesis, they span a rainbow $I \in \ci_{m-1}(G)$.
Since $I \subseteq V \setminus N[v]$, $I \cup \{v\}$ is a rainbow independent $m$-set.

\smallskip

\item
For any $k$ let 
  $G$ consist of  $\lceil \frac{n}{t} \rceil$ copies of $K_{t,t,\dots, t}$, the complete $k$-partite graph with sides of size $t$. Note $K_{1,t+1} \not < G$ as all components have $\alpha(K_{t,t,\dots, t})=t$. Let $I_i, ~i\le k$ consist of the union of all $i$-th independent $t$-tuples in all copies of $K_{t,t,\dots, t}$. Note $|I_i|=t \left\lceil \frac{n}{t} \right\rceil\geq n$. Any rainbow independent set contains at most one vertex from each component $K_{t,t,\dots, t}$, and hence cannot be larger than $\lceil \frac{n}{t} \rceil$.

\end{enumerate}
\end{proof}

%Recall that $K_r^-$ denotes $K_r$ minus an edge, 
%so removing a single leaf from the claw gives $K_3^-$.
%In sharp contrast to the claw, 
%$K_3^-$-free graphs need the smallest number of independent sets possible:

When $m=n \geq 3$, $t=2$ gives $f_{\cf(K_{1,3})}(n,n)=\infty$. But $t=1$ gives:
 
\begin{observation}\label{t.disjointcliques}
%For $n \geq m$,
\[
f_{\cf(K_3^-)}(n,n)=n.
\]
\end{observation}

%\begin{proof}
%The non-adjacency relation $\not\sim$ on $V(H)$ is always reflexive and symmetric. $K_3^{--} \not < H$ says precisely that $\not\sim$ is transitive, and $H$ being complete multipartite says precisely that $\not\sim$ is an equivalence relation.  For the second part, apply the first part to the graph complement $H:=G^c$.
%\end{proof}

\begin{comment}
\begin{proof}[Proof of Theorem \ref{t.disjointcliques}.]
By \eqref{trivialbound}, it suffices to show $f_{\cf(K_3^-)}(n,m) \leq m$.
Suppose $K_3^- \not < G$.
Then $\ci(G)$ is a partition matroid whose parts are the cliques of $G$ from Lemma \ref{lem multipartite}. Greedy choice shows that in a matroid every $m$ independent $n$-sets have a rainbow independent $m$-set. 
\end{proof}
\end{comment}

In fact, non-finiteness is the rule, and $K_3^-$ is one of few exceptions. 
%Recall that $K_r^-$ denotes $K_r$ minus an edge.
The main result of this section is:
\begin{theorem}
\label{t.forbidden_introduction} 
% Let $n \ge 4$. Then $f_{\cf(H)}(n,n) < \infty$ if and only if  $ H \simeq K_r$ or $ K_{r}^-$ for some $r$.
 $f_{\cf(H)}(n,n) < \infty$
 for every positive integer $n$ if and only if $H$ is either $K_r$ or $K_{r}^-$ for some $r$.
\end{theorem}

To connect the theorem to Theorem \ref{gendrisko}, note that line graphs are 
$K_5^-$-free. In fact, $K_5^-$ is one of nine   forbidden subgraphs, including the claw, whose exclusion as induced subgraphs characterizes line graphs - see \cite{beineke}. In this context, let us note two more facts:

\begin{itemize}
\item A graph $H$ is the line graph of a triangle-free graph if and only if $K_{1,3}, K_4^- \not < H$ \cite{Ar}.
\item A graph $H$ is the line graph of a bipartite graph if and only if $K_{1,3}, K_4^-, C_5, C_7, C_9, \dots \not < H$  \cite{He}.
\end{itemize}

We start the proof of Theorem \ref{t.forbidden_introduction}
by showing necessity, namely:\\ 

($\diamondsuit$) If $H$ is neither $K_r$ nor $K_r^-$ for any $r$, then  $f_{\cf(H)}(n,n)= \infty$. \\

%\subsection{Claw-free graphs}\label{s.clawfree}

We have already shown ($\diamondsuit$) for $H$ being the claw. This will be used below. Next we consider the case of $C_4$.

%\subsection{$H$-free graphs}
%\subsection{$K_r$-free graphs and $K_r^-$-free graphs}
%\section{$K_r$ and $K_r^-$-free graphs}

\begin{lemma}\label{l.C4freeinfinite}
 If $n\ge 4$ then $f_{\cf(C_4)}(n,n)= \infty$.
 \end{lemma}

The proof of the lemma  uses a construction generalizing Example \ref{driskosharp}.

%Recall that Theorem \ref{t.forbidden_introduction} already established that both $f_{\cf(K_r^-)}(n,m)$ and $ f_{\cf(K_r)}(n,m)$ were finite.

%To establish  upper bounds
%on $f_{\cf(K_r^-)}(n,m)$, we shall use a Ramsey-type result on multipartite graphs.
%Again, recall $R(a,b)$ denotes the smallest number $q$ such that if the edges of $K_q$ are coloured red and blue, there necessarily exist a red $K_a$ or a blue $K_b$. 

\begin{example}\label{mainconstruction}
For integers $n,t$,
let $G_{t,n}$ be obtained from a cycle of length $tn$ by adding all edges connecting any two vertices of distance smaller than $t$ in the cycle.
There are precisely $t$ independent $n$-sets, say $I_1, \ldots , I_{t}$.
As the family of independent sets, take $n-1$ copies of each $I_j$, yielding $t(n-1)$ colours in total.
(Setting $t=2$ gives Example \ref{driskosharp}.)

%that is, for each $1 \leq j \leq t$, we let $I_j = I_{t + j} =I_{2t + j} = \cdots = I_{(n-2)t + j}$. 
\end{example}

Since $\ci_n(G_{t,n})=\{I_1, \dots, I_t\}$ alone, and each $I_j$ repeats only $n-1$ times, $G_{t,n}$ has no rainbow independent $n$-set. Hence the claim will be proved if we show $C_4 \not < G_{t,n}$ for $n \geq 4$.

Suppose, for contradiction, that $x_1,x_2,x_3, x_4$ form an induced $C_4$ in $G_{t,n}$. Then their consecutive distances $d_i:=x_{i+1}-x_{i}$ take values in $\{\pm 1, \pm 2, \dots, \pm(t-1) \}$, since $x_i$ and $x_{i+1}$ are adjacent. Furthermore, $d_i, d_{i+1}$ have the same sign, as otherwise $|x_{i+2}-x_i| =|d_{i+1}-d_i| <t-1$,  contradicting the non-adjacency of $x_i$ and $x_{i+2}$.
Without loss of generality, $d_i>0$ for every $i$.
Then $x_1 \equiv x_2 -d_1 \equiv x_3 - d_2 - d_1 \dots \equiv x_1 - \sum_i d_i \pmod {nt}$. So
$\sum_i d_i$ is a positive multiple of $nt$, hence at least $nt$. But 
\[
\sum_{i=1}^4 d_i \leq \sum_{i=1}^4 (t-1) \leq 4(t-1)<nt,
\]
a contradiction.

%Note that for any $s \geq 4$ we have $f_{\cf(C_4, C_5, \dots, C_{s})}(2,2) = 2$.
%But for  $n \geq 3$, forbidding  $C_{n+1}$ changes the picture:
%%\begin{theorem}\label{C4 free}
%% $f_{\cf(C_4)}(3,3) <\infty$.
%%\end{theorem}
%
%\begin{theorem}\label{t.C4-Csfreefinite}
%If $s > n \geq 3$ then 
%$f_{\cf(C_4, C_5, \dots, C_{s})}(n,n)  < \infty$.
%\end{theorem}

%In fact, when $m=n$, the graphs $K_r$ and $K_r^-$ are the only
% (nonempty)
%graphs for which Theorem \ref{t.forbidden_introduction} holds. We prove this in a series of lemmas. 

Thirdly, let $K_3^{--}$ be the graph on three vertices with exactly one edge.
\begin{lemma}\label{lem multipartite}{(Folklore)}
$K_3^{--} \not < H$ if and only if $H$ is a complete $r$-partite graph $K_{s_1, s_2, \dots, s_r}$ for some $r$.
%Equivalently, $K_3^- \not < G$ if and only if $G$ is a disjoint union of cliques. 
\end{lemma}

\begin{lemma}
If $f_{\cf(H)}(n,n) < \infty$ for every $n$, then:
\begin{enumerate}[(a)]
    \item\label{aaa} $H$ is claw-free,
    \item\label{bbb} $H$ is $C_4$-free.
%    \item\label{ccc} $H$ is $\bar{K}_{n+1}$-free, and 
    \item\label{ddd} $H$ is $K_3^{--}$-free.
\end{enumerate}
\end{lemma}

\begin{proof}
%%
%% @@ Removed $\bar{K}_{n+1}$ from here, no longer seems relevant?
%%
\eqref{aaa} and \eqref{bbb} follow from Observation~\ref{montonicity2} combined with Theorem~\ref{clawfree} and Lemma~\ref{l.C4freeinfinite}.
For % \eqref{ccc} and 
\eqref{ddd}, consider the complete $t$-partite graph $G = K_{n, \ldots, n}$.
%Let $K_{n,t}$ be the complete multipartite graph on $t$ disjoint colour classes $I_1, \dots, I_t$ of size $n$.
By Lemma \ref{lem multipartite}, $G$ is
%$\bar{K}_{n+1}$ by pigeonhole principle, nor 
$K_3^{--}$-free.
On the other hand, 
%$K_{n, \dots, n}$
$G$
has $t$ pairwise disjoint independent $n$-sets spanning no rainbow independent $2$-set (let alone an $n$-set).
Choosing $t$ is arbitrarily large proves % that $f_{\cf(\bar{K}_{n+1})}(n,n) = \infty$ and
that $f_{\cf(K_3^{--})}(n,n)=\infty$ for every $n\geq2$.
Hence % \eqref{ccc} and 
\eqref{ddd} follows from Observation~\ref{montonicity2}.
\end{proof}

Part (c) and Lemma \ref{lem multipartite} imply: 

\begin{lemma}
If $f_{\cf(H)}(n,n) < \infty$, then
$H$ is complete multipartite. 
\end{lemma}

%We can now summarise the results of this section so far as follows: 
%\begin{theorem}\label{t.onlyKrKr-} 
%If $f_{\cf(H)}(n,n) < \infty$ for every positive integer $n$ then  $H$ is either $K_r$ or $K_{r}^-$ for some $r$.
%\end{theorem}

\begin{proof}[Proof of ($\diamondsuit$).]
We have shown that if $f_{\cf(H)}(n,n) < \infty$ then $H$ is multipartite, avoiding $C_4$ and the claw as induced subgraphs. $C_4 \not < H$
implies that at most one class in the partition of the graph is of size $2$ or more, and
%the absence of induced
$K_{1,3} \not < H$ implies that if there is a class of size larger than $1$ it is of size $2$. If there is no such class, $H$ is complete. If there is a single class of size $2$, then $H$ is $K_r^-$. 
\end{proof}

In order to prove the other direction of Theorem \ref{t.forbidden_introduction}, we introduce a variant of $f_\cc(n,m)$.

\begin{definition}\label{def.fdisjoint}
Let $\f_\cc(n,m)$ be the minimal number $k$ such that every $k$ {\em disjoint} independent $n$-sets in any graph belonging to $\cc$ have a partial rainbow independent $m$-set.
\end{definition}
Clearly, $\f$, like $f$, satisfies \eqref{trivialbound} and \eqref{monotonepartial}. It is also clear that  
\begin{equation}\label{eq.sunflower reduction}
\f_{\cc}(n,m) \leq 
    f_{\cc}(n,m).
    \end{equation}

The next theorem establishes equivalence between the finiteness of $f_\C$ for all values of $n,m$ and the finiteness of $f'_\C$ for all values of $n,m$.

\begin{theorem}\label{t.sunflower reduction}
If $m\le n$ then
$$ f_{\cc}(n,m)\le    n!\big(\max_{\ell \leq m}\{\f_{\cc}(n-m+\ell,\ell)\}-1\big)^n.
$$
\end{theorem}

 For the proof, we  recall some notions from Ramsey theory.

A {\em sunflower} is a collection of sets $S_1,\ldots, S_k$ with the property that, for some set $Y$, $S_i \cap S_j = Y$ for every pair $i \neq j$.
The set $Y$ is called the {\em core} of the sunflower and the sets $S_i \setminus Y$ are called {\em petals}.
In particular, a collection of pairwise disjoint sets is a sunflower with $Y = \emptyset$.

\begin{lemma}[Erd\H{o}s-Rado Sunflower Lemma,  \cite{ER}]\label{sunflower lemma}
Any collection of  $n! (k-1)^n$ sets of cardinality $n$ contains a sunflower with $k$ petals.
\end{lemma}
It is worth remarking that the number of sets needed has been reduced in \cite{kostochka}, and again (very recently) in \cite{alwz}. The {\em sunflower conjecture}, stating that $c_k^n$ sets of cardinality $n$ may suffice, is still open. 

%We use the sunflower lemma to reduce Theorem \ref{C4-Cs free} to the case in which all independent sets are disjoint.

%\begin{theorem}\label{t.C4-Cs disjoint free}
%For all numbers $n \geq 3$ there is some large $N_n$, increasing with $n$, satisfying the following property.
%Suppose $s > n$ and $G \in \cf(C_4, \dots, C_s)$ has $N_n$ \emph{disjoint} independent $n$-sets.
%Then $G$ contains a rainbow independent $n$-set.
%\end{theorem}

%\begin{proof}[Proof of Theorem \ref{C4-Cs free}]
\begin{proof}[Proof of Theorem~\ref{t.sunflower reduction}]

%We claim that $f_{\cf(C_4, \dots, C_s)}(n,n) \leq
%n!(N_n-1)^n$.
For $a \ge b$ let $N_a^b=\max (\f_\cc(a,b),m)$. Let $N=\max \{N_{n-m+\ell}^{\ell} :
\ell \leq m \}$.

We proceed to prove that $f_\cc(n,m) \leq n!(N-1)^n$, which implies the theorem since \eqref{trivialbound} gives $m \leq \f_\cc(n,m) \leq N$.
Let $I_1, \dots, I_{n!(N-1)^n} \in \ci_n(G)$. By the sunflower lemma, some
$N$
of them form a sunflower $S_1, \dots, S_{N}$, say with core $Y$. If $\ell \leq n
%< s
$ is the size of the resulting petals $S_i \setminus Y$, then these
$S_i \setminus Y \in \ci_\ell(G)$ and are pairwise disjoint.
Plus, there are $N \geq N_\ell^{m-n+\ell}$ of them.

So applying 
%Theorem \ref{t.C4-Cs disjoint free}
the definition of $N_\ell^{n-m+\ell}=\f_\cc(\ell,m-n+\ell)$
to the induced graph $G[\cup_i (S_i \setminus Y)]$ gives a rainbow independent set $I$ of size $m-n+\ell$ among these petals $S_i \setminus Y$.
But extending $I$ to $I \cup Y$ also produces an independent set, now of size $m$, as the core $Y$ is nonadjacent to all vertices in the sunflower.
The additional $n -\ell$ vertices in $Y$ can all be assigned distinct new colours not used in $I$, since $N_\ell^{n-m+\ell}-|I| \geq m-(m-n+\ell)=n-\ell$ 
%from (1) 
 and $Y$ is contained in every $S_i$. So $I \cup Y$ is rainbow, as desired.
\end{proof}

%\begin{remark}
%It is not clear whether or not 
%\[
%\f_\cc(n,m)=\max_{\ell \leq m}\{ \f_\cc(n-m+\ell,\ell)\}
%\]
%in general, so it  may be difficult to simplify the upper bound in \eqref{eq.sunflower reduction}.
%
%@Simplification is not essential here, suggest deleting this remark. 
%\end{remark}

%We introduce a standard way to translate between graphs and digraphs.
In what follows, we allow digraphs to have loops and digons, but not parallel edges.

\begin{definition}
Let $\Gamma$ be a bipartite graph on vertex set $\{a,a'\} \times B$, whose parts are the two columns $\{a\} \times B$ and $\{a'\} \times B$.
Write $D(a,a')$ for the digraph on vertex set $B$, whose edges are given by
\[
bb' \in D(a,a') \Leftrightarrow (a,b) \sim (a',b') \text{ in } E(\Gamma).
\]
\end{definition}
%(For examples, see Figures \ref{fig.DAGcycle} and \ref{fig.twoditriangles}.)
For example:
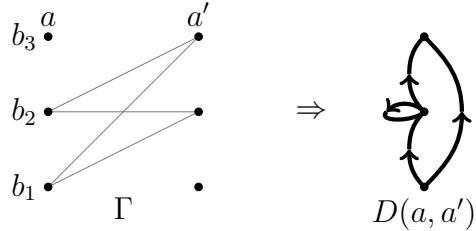
\begin{center}
\begin{tikzpicture}

\draw [gray](0,0) -- (2,1)--(0,1) -- (2,2)--(0,0);

\draw [ultra thick, ->] (5,0) to [out=135,in=270] (4.8,0.5);
\draw [ultra thick] (4.8,0.5) to [out=90,in=225] (5,1);
\draw [ultra thick, ->] (5,1) to [out=135,in=270] (4.8,1.5);
\draw [ultra thick] (4.8,1.5) to [out=90,in=225] (5,2);

\draw [ultra thick, ->] (5,0) to [out=45, in =270] (5.5,1);
\draw [ultra thick] (5.5,1) to [out=90, in=315] (5,2);

\draw [ultra thick, ->] (5,1) to [out=135, in=45] (4.5,1);
\draw [ultra thick] (4.5,1) to [out=225, in=225] (5,1);

\filldraw (0,0) circle [radius=0.05];
\filldraw (0,1) circle [radius=0.05];
\filldraw (0,2) circle [radius=0.05];

\filldraw (2,0) circle [radius=0.05];
\filldraw (2,1) circle [radius=0.05];
\filldraw (2,2) circle [radius=0.05];

\filldraw (5,0) circle [radius=0.05];
\filldraw (5,1) circle [radius=0.05];
\filldraw (5,2) circle [radius=0.05];

\node [below] at (1,0) {$\Gamma$};

\node [above] at (0,2) {$a$};
\node [above] at (2,2) {$a'$};

\node [below] at (5,0) {$D(a,a')$};

\node [left] at (0,0) {$b_1$};
\node [left] at (0,1) {$b_2$};
\node [left] at (0,2) {$b_3$};

\node at (3.5,1) {$\Rightarrow$};

\end{tikzpicture}
\end{center}

\begin{definition}
Let $A$ and $B$ be finite ordered sets with $|A| = N$.
Let $G$ be an $N$-partite graph on $A \times B$ whose parts are columns $\{a\} \times B$, and let $D$ be a digraph on $B$.
We say that $G$ is {\em repeating}, or {\em $D$-repeating}, if for every $a<a'$ the digraph $D(a,a')$ is the same digraph $D$ on $B$.
Equivalently:

For every $b_1, b_2$ in $B$ (not necessarily distinct) and two pairs $a_1 < a_2$ and $a_1' < a_2'$ in $A$, the vertices $(a_1, b_1)$ and $(a_2, b_2)$ are adjacent if and only if the vertices $(a_1', b_1)$ and $(a_2', b_2)$ are adjacent.
%That is, all bipartite graphs of the form $G[\{i, i'\} \times B]$ with $i \neq i'$ (induced by two parts of $G$) are isomorphic to a bipartite graph $\Gamma$ in a fashion consistent between parts.
%In particular, if 
%$G[\{i, i'\} \times B]$ is isomorphic to a bipartite graph $\Gamma$ for all $i \neq i'$,
%we say $G$ is {\em $D$-repeating}.
We say that $G$ is {\em strongly repeating} if $D$ has a loop at every vertex, so that all rows $G[A \times \{b\}]$ are cliques.
%%Suppose $A, B$ are finite ordered sets with $|A|=a$. Let $G$ be an $a$-partite graph on $A \times B$ whose parts are the columns $\{i\} \times B$. We say $G$ is \emph{repeating} if the following holds. 
%%
%%For every 2 pairs $i_1 < i_2, i_1' < i_2'$ in $A$ and $j_1, j_2 \in B$ (not necessarily distinct), we have $(i_1, j_1)$ and $(i_2, j_2)$ are adjacent if and only if $(i_1', j_1)$ and $(i_2', j_2)$ are. That is, all bipartite graphs $G[\{i,i'\} \times B]$ induced by 2 parts of $G$ are isomorphic in a fashion consistent between parts. We furthermore say $G$ is $\Gamma$-\emph{repeating} if $\Gamma$ is the induced graph on the first two columns of $G$.
\end{definition}

Note that every row $A \times \{b\}$ is either a clique or an independent set in a $D$-repeating graph $H$ on $A \times B$, depending on whether $D$ has a loop at vertex $b$ or not.
%and that the subgraph $H[A' \times B']$ obtained by inducing $H$ on the subgrid $A' \times B'$ of $A \times B$ is also repeating.

%Note that every row $A \times \{j\}$ is either a clique or an independent set in any repeating $H$, and that subgraphs $H[A' \times B']$ obtained by inducing $H$ on subgrids of $A \times B$ are also repeating.

%We next appeal to more Ramsey theory. 
Next recall that $R(r_1, \dots, r_c)$ denotes the smallest number of vertices in a complete graph for which any $c$-edge-colouring contains in some colour $i$ a monochromatic $K_{r_i}$.  Ramsey's theorem guarantees the existence of such a number.

\begin{lemma}\label{l.ramsey}
For every pair  $n, N$ of integers there is an $R$ with the following property.
Suppose $G$ is an $R$-partite graph on $[R] \times [n]$, whose parts are the $R$ columns.
Then there is some $A \subset [R]$ with $|A| =N$ for which the induced subgraph $G[A \times [n]]$ is repeating. 
\end{lemma}
\begin{proof}
We colour the edges of $K_R$ by assigning to each pair $a<a'$ the digraph $D(a,a')$. There are  $2^{n^2}$ such possible choices, and hence
\[
R:=R(\overbrace{N, \dots, N}^{2^{n^2}}).
\]

does the job.
\end{proof}

We can use this to specialise Definition \ref{def.fdisjoint} even further.
\begin{definition}\label{def.fr}
Let $\fr_\cc(n,m)$ be the minimal $k$ such that any {\em repeating} graph on $[k] \times [n]$ in $\cc$ has an independent $m$-set which is rainbow with respect to the $k$ columns $(\{a\} \times [n])_{a=1}^k$.
\end{definition}

As with $f$ and $\f$, $\fr$ also satisfies \eqref{trivialbound} and \eqref{monotonepartial}, and 
\begin{equation}\label{eq.repeating reduction}
            \fr_\cc(n,m) \leq \f_\cc(n,m).
\end{equation}

Lemma \ref{l.ramsey} yields:
\begin{theorem}\label{t.repeating reduction}
$$
 \f_\cc(n,m)
    \leq R(\underbrace{\fr_\cc(n,m), \dots, \fr_\cc(n,m)}_{2^{n^2}}).
    $$
\end{theorem}

We are now in a position to prove the main result of this section.

\begin{proof}[Proof of Theorem \ref{t.forbidden_introduction}]
The ``only if'' direction has already been shown in 
%Theorem \ref{t.onlyKrKr-}
($\diamondsuit$).
For the ``if'' direction, we show for $n \ge 2$ that
%\begin{itemize}
%    \item $\fr_{\cf(K_r)}(n,n) =\max\{n,r\}$, and
%    \item $ \fr_{\cf(K_r^-)}(n,n) = \max \{n,r-1\}$.
%\end{itemize}
\[
\fr_{\cf(K_r)}(n,n) =\max(n,r), \text{ and }
\fr_{\cf(K_r^-)}(n,n) = \max (n,r-1).
\]
Since both values above are  $0,1$ when $n=0,1$, respectively.
the result will then follow from Theorems \ref{t.sunflower reduction} and \ref{t.repeating reduction}.

Let us first prove
that $\fr_{\cf(K_r)}(n,n)\le \max(n,r)$ 
 and $\fr_{\cf(K_r^-)}(n,n) \le \max (n,r-1)$. 
Let $G$ be a repeating graph on $A \times [n]$. 

 If $G \in \cf(K_r)$, the first row $A\times \{1\}$ of $G$ is a rainbow independent set, provided $|A| \geq r$.

If $G \in \cf(K_r^-)$,
then again either some row $A \times \{b\}$ is empty, and hence is a rainbow independent set of size $\ge n$, or else all such are cliques, namely $G$ is strongly repeating.
Then the diagonal
\[
\{(a,a): 1 \leq a \leq n\}
\]
is a rainbow independent $n$-set. Otherwise, if some pair $a' < a$ have $(a',a')$ and $(a,a)$ adjacent in $G$, then $(a',1) \sim (a,b)$ for every $b \geq 2$ by the repeating property (see Figure \ref{fig.K_5^-}). Then $K_r^- < G$ as witnessed by
\[
\{(a',1)\} \cup \{(a,b):1\leq b \leq r-1\},
\]
where $(a',1) \not\sim (a,1)$ is the missing edge-a contradiction.

\begin{figure}
\begin{center}
    \begin{tikzpicture}[scale=2]
%    \draw [fill=gray!10];
\draw [gray] (1,2)--(0,1)--(2,2)--(1,1)--(3,2)--(2,1);
\draw [gray] (0,1)--(3,2);
\draw[gray]
(1,0)--(0,0)
to [out=345,in=205](2,0)--(1,0)
to [out=345,in=205](3,0)--(2,0);
\draw [gray] (0,0) to [out=345,in=205] (3,0);

\draw[gray]
(1,1)--(0,1)
to [out=345,in=205](2,1)--(1,1)
to [out=345,in=205](3,1)--(2,1);
\draw [gray] (0,1) to [out=345,in=205] (3,1);

\draw[ultra thick]
(1,2)--(0.05,2)
to [out=15,in=165](2,2)--(1,2)
to [out=15,in=165](2.95,2)--(2,2);
\draw [ultra thick] (0.05,2) to [out=15,in=165] (2.95,2);

\draw [thick] (1,1)--(2,2);

\draw [ultra thick](0,1)--(1,2);
\draw [ultra thick](0,1)--(2,2);
\draw [ultra thick](0,1)--(3,2);

\node(00) at (0,0){};
\node(01) at (0,1){};
\node(02) at (0,2){};

\node(10) at (1,0){};
\node(11) at (1,1){};
\node(12) at (1,2){};

\node(20) at (2,0){};
\node(21) at (2,1){};
\node(22) at (2,2){};

\node(30) at (3,0){};
\node(31) at (3,1){};
\node(32) at (3,2){};

\filldraw (00) circle [radius=0.05];
\filldraw (01) circle [radius=0.05];
\filldraw (02) circle [radius=0.05];

\filldraw (10) circle [radius=0.05];
\filldraw (11) circle [radius=0.05];
\filldraw (12) circle [radius=0.05];

\filldraw (20) circle [radius=0.05];
\filldraw (21) circle [radius=0.05];
\filldraw (22) circle [radius=0.05];

\filldraw (30) circle [radius=0.05];
\filldraw (31) circle [radius=0.05];
\filldraw (32) circle [radius=0.05];

\draw [rounded corners=14pt]
(-0.3,-0.3) rectangle (0.3,2.3);
\draw [rounded corners=14pt]
(0.7,-0.3) rectangle (1.3,2.3);
\draw [rounded corners=14pt]
(1.7,-0.3) rectangle (2.3,2.3);
\draw [rounded corners=14pt]
(2.7,-0.3) rectangle (3.3,2.3);

\draw[dashed, gray, thick, rounded corners=10pt
,rotate=45
] (-0.2,-0.2) rectangle (3.05,0.2);

    \end{tikzpicture}
\caption{$K_5^-$ in a strongly repeating graph where $(2,2) \sim (3,3)$.
%If the top row were not a clique, it would be a rainbow independent set.
\label{fig.K_5^-}}
\end{center}
\end{figure}
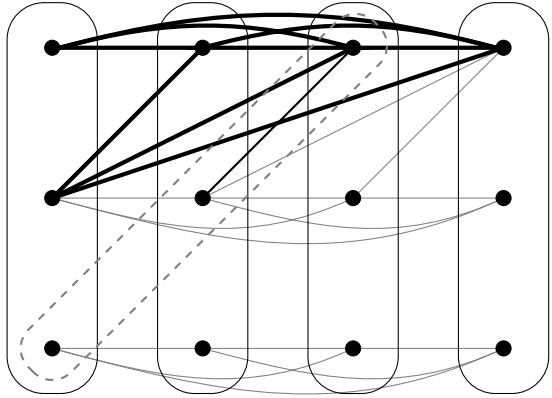

To prove the corresponding lower bounds, we show separately that 
$\fr_{\cf(K_r^-)}(n,n) \geq r-1$ and $\geq n$ ($\fr_{\cf(K_r)}(n,n) \geq \max\{n,r\}$ follows similarly). 
The first is witnessed by the complete $(r-2)$-partite graph with $n$ vertices in each part.
%(and only one copy of each independent $n$-set), 
This is repeating, $K_r^-$-free, and
has no rainbow independent $n$-set
provided $n\geq 2$.
%As this graph is repeating, we have $\fr_{\cf(K_r^-)}(n,n) \geq r-1$. Similarly, $\fr_{\cf(K_r)}(n,n) \geq r$.
The second follows from \eqref{trivialbound} for $\fr$.  

\end{proof}

To complete this section, we use the Ramsey numbers $R(s,t)$ to classify $f_{\cf(K_{r})}$ and find a nontrivial lower bound on $f_{\cf(K_{r+1}^-)}$.
%Recall again that the \emph{Ramsey number} $R(r,t)$ is the smallest $N$ for which any graph on $N$ vertices contains either a $K_r$ or a $\bar{K}_t$.

\begin{theorem}\label{t.ramseyequality}
For any numbers $r$ and $m \leq n$:
 \[R(r,m) = f_{\cf(K_r)}(n,m) \leq f_{\cf(K_{r+1}^-)}(n,m).\]
\end{theorem}

\begin{proof}
We first find a graph $G$ showing $f_{\cf(K_r)}(n,m) > N:= R(r,m)-1$ (see Figure \ref{fig.ramseyrainbowexample}).
%\Omega\left( \frac{m^2}{\log m}\right) 
%for the family $\mathcal{F}(K_3)$ of triangle-free graphs.
%($ f_{\mathcal{F}(K_4^-)}(n,m)
%\geq f_{\mathcal{F}(K_3)}(n,m)$ 
%follows since $K_3$-free $\Rightarrow K_4^-$-free.)
Take a $K_r$-free graph $H$ on $N$ vertices with $\alpha(H)<m$, as guaranteed by the definition of $R(r,m)$.
Let $G$ be the graph blowup $H^{(n)}$. That is, replace each $v\in V(H)$ by an independent $n$-set, and replace each edge in $H$ by the corresponding complete bipartite graph in $G$ (specifically a copy of $K_{n,n}$). Then $K_r \not < G$ since $K_r \not < H$. Letting $I_1, \dots, I_N$ be the $N$ blown up vertices yields no rainbow independent $m$-set in $G$, since the $I_j$'s are disjoint and $\alpha(H)<m$.

To show that $f_{\cf(K_r)}(n,m) \leq R(r,m)$, let $G$ be a $K_r$-free graph and $I_1,\ldots,I_{R(r,m)} \in \ci_n(G)$.
Let $M$ be an inclusion-maximal rainbow set.
If $M$ represents all sets $I_j$, then $|M|=R(r,m)$.
Since $K_r \not < G$, it must be that $M$ contains an independent $m$-set, which is rainbow as required.
Thus we may assume that some $I_j$ is not represented in $M$.
By the maximality of $M$, this implies that  $M \supseteq I_j$, implying in turn that $I_j$ is a rainbow independent set of size $n\ge m$. 

The right-hand inequality is due to the monotonicity expressed in  \eqref{monotonepartial}, since the fact that $K_r \subseteq K_{r+1}^-$ implies that $\cf(K_r) \subseteq\cf(K_{r+1}^-)$.
\end{proof} 

\begin{figure}
\begin{center}
\begin{tikzpicture}[thick, scale=0.8]
\begin{scope}[xshift=-6]

\draw[white, dashed, rounded corners=8pt] (-0.5,1) rectangle (0.5,3.5);
\draw[white, dashed, rounded corners=8pt,rotate=72] (-0.5,1) rectangle (0.5,3.5);
\draw[white, dashed, rounded corners=8pt,rotate=144] (-0.5,1) rectangle (0.5,3.5);
\draw[white, dashed, rounded corners=8pt,rotate=216] (-0.5,1) rectangle (0.5,3.5);
\draw[white, dashed, rounded corners=8pt,rotate=288] (-0.5,1) rectangle (0.5,3.5);

\filldraw (90:2.2) circle(1pt);
\filldraw (162:2.2) circle(1pt);
\filldraw (234:2.2) circle(1pt);
\filldraw (306:2.2) circle(1pt);
\filldraw (18:2.2) circle(1pt);

\draw (90:2.2)--(162:2.2);
\draw (162:2.2)--(234:2.2);
\draw (234:2.2)--(306:2.2);
\draw (306:2.2)--(18:2.2);
\draw (18:2.2)--(90:2.2);

%\node [white] at (0,-3.5) {ghost};
\node at (0,-3.5){$R(3,3) > 5$};
\end{scope}
\end{tikzpicture}
\begin{tikzpicture}[thick, scale=0.8]
\node at (-3.5,0) {$\longrightarrow$};
\node at (-3.5,-3.5) {$\Rightarrow$};
\end{tikzpicture}
\begin{tikzpicture}[thick, scale=0.8]

\node at (0,-3.5){$f_{\cf(K_3)}(4,3) > 5$.};

\filldraw (90:2) circle(1pt);
\filldraw (162:2) circle(1pt);
\filldraw (234:2) circle(1pt);
\filldraw (306:2) circle(1pt);
\filldraw (18:2) circle(1pt);

\filldraw (90:3) circle(1pt);
\filldraw (162:3) circle(1pt);
\filldraw (234:3) circle(1pt);
\filldraw (306:3) circle(1pt);
\filldraw (18:3) circle(1pt);

\draw (90:2)--(162:2);
\draw (162:2)--(234:2);
\draw (234:2)--(306:2);
\draw (306:2)--(18:2);
\draw (18:2)--(90:2);

\draw (90:2)--(162:3);
\draw (162:2)--(234:3);
\draw (234:2)--(306:3);
\draw (306:2)--(18:3);
\draw (18:2)--(90:3);

\draw (90:3)--(162:3);
\draw (162:3)--(234:3);
\draw (234:3)--(306:3);
\draw (306:3)--(18:3);
\draw (18:3)--(90:3);

\draw (90:3)--(162:2);
\draw (162:3)--(234:2);
\draw (234:3)--(306:2);
\draw (306:3)--(18:2);
\draw (18:3)--(90:2);

\filldraw (90:2.5) circle(1pt);
\filldraw (162:2.5) circle(1pt);
\filldraw (234:2.5) circle(1pt);
\filldraw (306:2.5) circle(1pt);
\filldraw (18:2.5) circle(1pt);

\filldraw (90:1.5) circle(1pt);
\filldraw (162:1.5) circle(1pt);
\filldraw (234:1.5) circle(1pt);
\filldraw (306:1.5) circle(1pt);
\filldraw (18:1.5) circle(1pt);

\draw (90:2.5)--(162:2.5);
\draw (162:2.5)--(234:2.5);
\draw (234:2.5)--(306:2.5);
\draw (306:2.5)--(18:2.5);
\draw (18:2.5)--(90:2.5);

\draw (90:2.5)--(162:2);
\draw (162:2.5)--(234:2);
\draw (234:2.5)--(306:2);
\draw (306:2.5)--(18:2);
\draw (18:2.5)--(90:2);

\draw (90:2)--(162:2.5);
\draw (162:2)--(234:2.5);
\draw (234:2)--(306:2.5);
\draw (306:2)--(18:2.5);
\draw (18:2)--(90:2.5);

\draw (90:2.5)--(162:3);
\draw (162:2.5)--(234:3);
\draw (234:2.5)--(306:3);
\draw (306:2.5)--(18:3);
\draw (18:2.5)--(90:3);

\draw (90:3)--(162:2.5);
\draw (162:3)--(234:2.5);
\draw (234:3)--(306:2.5);
\draw (306:3)--(18:2.5);
\draw (18:3)--(90:2.5);

\draw (90:1.5)--(162:2.5);
\draw (162:1.5)--(234:2.5);
\draw (234:1.5)--(306:2.5);
\draw (306:1.5)--(18:2.5);
\draw (18:1.5)--(90:2.5);

\draw (90:1.5)--(162:1.5);
\draw (162:1.5)--(234:1.5);
\draw (234:1.5)--(306:1.5);
\draw (306:1.5)--(18:1.5);
\draw (18:1.5)--(90:1.5);

\draw (90:2.5)--(162:1.5);
\draw (162:2.5)--(234:1.5);
\draw (234:2.5)--(306:1.5);
\draw (306:2.5)--(18:1.5);
\draw (18:2.5)--(90:1.5);

\draw (90:1.5)--(162:2);
\draw (162:1.5)--(234:2);
\draw (234:1.5)--(306:2);
\draw (306:1.5)--(18:2);
\draw (18:1.5)--(90:2);

\draw (90:2)--(162:1.5);
\draw (162:2)--(234:1.5);
\draw (234:2)--(306:1.5);
\draw (306:2)--(18:1.5);
\draw (18:2)--(90:1.5);

\draw (90:1.5)--(162:3);
\draw (162:1.5)--(234:3);
\draw (234:1.5)--(306:3);
\draw (306:1.5)--(18:3);
\draw (18:1.5)--(90:3);

\draw (90:3)--(162:1.5);
\draw (162:3)--(234:1.5);
\draw (234:3)--(306:1.5);
\draw (306:3)--(18:1.5);
\draw (18:3)--(90:1.5);

\draw[dashed, rounded corners=8pt] (-0.5,1) rectangle (0.5,3.5);
\draw[dashed, rounded corners=8pt,rotate=72] (-0.5,1) rectangle (0.5,3.5);
\draw[dashed, rounded corners=8pt,rotate=144] (-0.5,1) rectangle (0.5,3.5);
\draw[dashed, rounded corners=8pt,rotate=216] (-0.5,1) rectangle (0.5,3.5);
\draw[dashed, rounded corners=8pt,rotate=288] (-0.5,1) rectangle (0.5,3.5);

\end{tikzpicture}
%$R(3,3)>5 \Rightarrow f_{\cf(K_3)}(4,3)>5$.
\end{center}
\caption{One direction of the equality in Theorem \ref{t.ramseyequality}.}\label{fig.ramseyrainbowexample}
\end{figure}
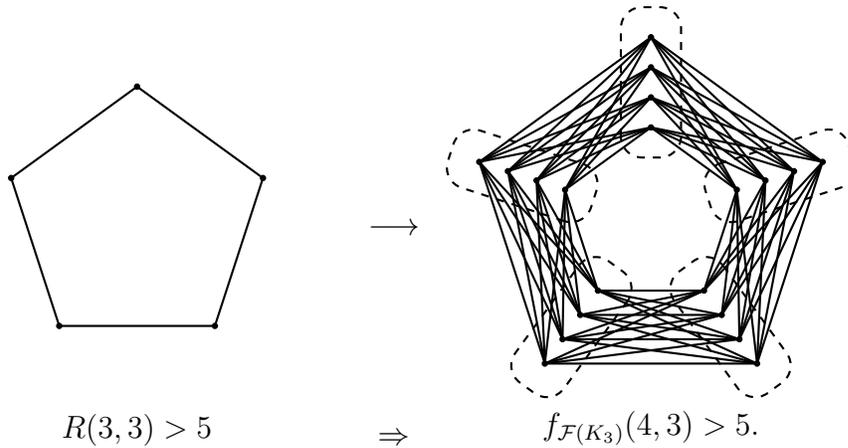

\begin{remark}\label{rem.disjoint Kr-free}
Since the construction above uses disjoint independent sets, this in fact shows
\[
\f_{\cf(K_r)}(n,m) = f_{\cf(K_r)}(n,m) \; (=R(r,m)).
\]
%using the notation from Definition \ref{def.fdisjoint}.
\end{remark}

%\subsection{$\{C_4, C_5, \dots, C_s\}$-free graphs}
\subsection{Chordal graphs and $\{C_4, C_5, \dots, C_s\}$-free graphs}

A graph $G$ is called {\em chordal} if 
%it does not contain an induced cycle of length larger than $3$. 
$C_s \not < G$ for any $s\geq 4$.
Recall that the class of chordal graphs is denoted by $\ct$. 

\begin{theorem}\label{t.chordal rb}
If   $m \le n$ then $f_\ct(n,m)=m.$
\end{theorem}

This will follow from the following basic  property of chordal graphs:
\begin{theorem}\label{chordal}  (\cite{ChL}, Theorem 8.11). 
Any chordal graph contains a simplicial vertex, namely a vertex whose neighbors form a clique.
\end{theorem}

\begin{proof}[Proof of Theorem \ref{t.chordal rb}.]
By \eqref{trivialbound} it suffices to show that $f_\ct(n,m) \le m$. The proof is by induction on $m$. For $m=0$ there is nothing to prove. Assume the result is valid for $m-1$. 
Let $G$ be a chordal graph and let $I_1,\ldots,I_m \in \ci_n(G)$.
Let $V = I_1 \cup \cdots \cup I_m$.

Let $v$ be a simplicial vertex. 
Without loss of generality, $v \in I_m$.
Consider the induced subgraph  $G' = G[V\setminus N[v]]$ and the $m-1$ independent sets $I_j' = I_j \setminus N[v]$, $1 \leq j \leq m-1$, in $G'$.
Since $N[v]$ is a clique, any independent set in $G$ contains at most one vertex from $N[v]$, hence each $I_j'$ has cardinality at least $n-1$.
Since $G'$ is also chordal, by induction we may assume that there is a rainbow independent set $\{v_1, \ldots, v_{m-1}\}$ in $G'$, where $v_j \in I_j$ for each $j$.
Since $v$ is not adjacent to any $v_j$, the set $\{v_1,\ldots,v_{m-1}, v\}$ is a rainbow independent set in $G$.
%\hfill{\square}
\end{proof}

Chordal graphs exclude, as induced subgraphs, all cycles of length $\ge 4$. Given the  contrast between chordal graphs in Theorem \ref{t.chordal rb} and $C_4$-free graphs in Lemma \ref{l.C4freeinfinite}, it is of interest to know what happens if we exclude cycles up to a certain length.

We shall write $\Ck$ for a directed cycle of length $k\geq 2$ in a digraph. 
In this notation, $\overrightarrow{C_2}$ is a digon. 

%Recall that a ``DAG'' is a \emph{directed acyclic graph}, that is, a digraph containing no $\Ck$ for any $k$.

\begin{theorem}\label{t.DAGgeneral}
Let $s\geq 4$, and $m \leq n$. The following are equivalent:

%\begin{enumerate}[a.]
a.
There exists a directed graph $D$ on $n$ vertices with no undirected cycle of length $\leq s$ except for $\Cs$,
    nor any acyclic set  on $m$ vertices.

 %   \item
b. $f_{\F(C_4, \dots, C_s)}(n,m) = \infty$.
%    \item $f_{\F(C_4, \dots, C_s)}(n,m) \geq \max (m,s)$,
%    \item 
    %There exists a directed graph $D$ on $n$ vertices 
%satisfying the following:
%
%(i)~~Every set $S$ of $m$ vertices contains a directed cycle, and 
%
%(ii) If $C$ is an  undirected cycle of length $\le s$, then its directed edges form a $\Cs$.
%   

%\end{enumerate}
\end{theorem}

By an undirected cycle, we mean a cycle in the underlying (undirected) graph.

Note that, when $n=m=s$, $D:=\Cn$  itself is not acyclic and has no smaller cycle, so Theorem \ref{t.DAGgeneral} shows \[f_{\F(C_4, \dots, C_n)}(n,n)=\infty. \] By  \eqref{monotonepartial}, this implies Lemma \ref{l.C4freeinfinite} (indeed, $G_{t,n}$ can be viewed as a $\Cn$-repeating graph, and $G_{t,n} \in \F(C_4, \dots, C_n)$).
On the other hand, any $D$ on $m=n$ vertices which is not acyclic  contains a $\Ck$ on some $k < n+1$ vertices. So Theorem \ref{t.DAGgeneral} yields
\[
f_{\F(C_4, \dots, C_{n+1})}(n,n)<\infty. \]
%also giving Theorem \ref{t.C4-Csfreefinite} by monotonicity.

As in the proof of Theorem \ref{t.forbidden_introduction}, we use a reduction to the repeating version:

\begin{theorem}\label{t.DAGrepeating}
Suppose $s\geq 4,n,m$ are natural numbers, and $m \leq n$. Then the following are equivalent:

%\begin{enumerate}[(a)]
a.    There exists a directed graph $D$ on $n$ vertices with no undirected cycle of length $\leq s$ except for $\Cs$,
    nor any acyclic set  on $m$ vertices.
    
(b)    $\fr_{\F(C_4, \dots, C_s)}(n,m) = \infty$.
    
(c)    $\fr_{\F(C_4, \dots, C_s)}(n,m) \geq \max (m,s)$.
%\end{enumerate}
\end{theorem}

We first show how Theorem \ref{t.DAGgeneral} reduces to Theorem~\ref{t.DAGrepeating}.

\begin{proof}[Proof of Theorem~\ref{t.DAGgeneral}]
\underline{$b. \Rightarrow a.$}:
Suppose $f_{\F(C_4, \dots, C_s)}(n,m) = \infty$.
By Theorems \ref{t.sunflower reduction} and \ref{t.repeating reduction}, some $\ell \leq m$ has $\fr_{\F(C_4, \dots, C_s)}(n-m+\ell, \ell) = \infty$.
Theorem \ref{t.DAGrepeating} then gives a digraph $D'$ on $n-m+\ell$ vertices containing no cycle of length $\leq s$ except $\Cs$, and with no acyclic set of size $\ell$.

To conclude, form $D$ from $D'$ by adding $m-\ell$ isolated vertices. This way, $D$ has $n$ vertices and no additional cycles, and crucially no acyclic set of size $m$.

\underline{$a. \Rightarrow b.$} is immediate from Theorem \ref{t.DAGrepeating} together with \eqref{eq.sunflower reduction} and \eqref{eq.repeating reduction}.
\end{proof}

So, it remains to prove the repeating case.

\begin{proof}[Proof of Theorem~\ref{t.DAGrepeating}]
\underline{(b) $\Rightarrow$ (c)} is trivial.

\underline{(c) $\Rightarrow$ a.}:
Assuming (c), let $N=\max(m,s)$.
Take a repeating graph $G\in \cf(C_4, \dots, C_s)$ on $[N] \times [n]$, with no rainbow independent $m$-set with respect to the columns $\{a\} \times [n]$.

%Writing $\Gamma:=G[\{1,2\} \times [n]]$, and
As $N \geq m$, $G$ is strongly $D$-repeating for some $D$, i.e. $D$ has a loop at every vertex (otherwise some row contains a rainbow independent $m$-set).

%Define a digraph $D$ with vertex set $[n]$ as follows:
%\[
%b b' \in E(D) \qquad \Leftrightarrow \qquad
%b \neq b' \text{ and } (1,b) \sim (2,b') \text { in } E(G).
%\]
We next claim that $D$ is as desired.
First, assume for contradiction that $B'=\{b_1, \dots, b_m\}$ is an induced acyclic subgraph in $D$.
Then $D[B']$ can be completed to a transitive tournament; and in particular $B'$ can be relabelled so that
\[
\forall i, j \in [m], \qquad i > j \qquad \Rightarrow \qquad b_i b_j \not\in E(D).
\]
But this means that $\{(1,b_1), (2,b_2), \dots, (m,b_m)\}$ is a rainbow independent set in $G$, a contradiction. 

Next, suppose $D$ has a cycle of length $\leq s$ other than $\Cs$.
Then for some $s' \leq s$ there is an {\em induced} cycle $C=b_1 b_2 \dots b_{s'}b_1$ in $D$ (also not $\Cs$). We will use $C$ to find a cycle in $G$ of length between $4$ and $s$, giving the desired contradiction.

{\em Case 1.} $s' \in [4,s]$ and $C$ is not a directed cycle.
Then $C$ is acyclic, and as above there is a total ordering on $C$ so that all edges are oriented forwards, namely a relabelling $a_1, \dots, a_{s'}$ of $\{1, \dots, s'\}$ so that 
\[
\forall i\in [s']
\left\lbrace
\begin{array}{rl}
& b_i b_{i+1} \in E(D) \Rightarrow a_i < a_{i+1}, \text{ and} \\
& b_{i+1} b_i \in E(D) \Rightarrow a_i > a_{i+1}.
\end{array}
\right.
\]

Then
\[
(a_1, b_1), (a_2,b_2), \dots, (a_{s'}, b_{s'}), 
(a_1, b_1)
\]
is a copy of $C_{s'}$ in $G$ (see Figure \ref{fig.DAGcycle}),
%as there are enough columns (namely $N \geq s \geq s'$).
since $N \geq s \geq s'$.

%
%
% FIRST PICTURE ILLUSTRATING DAG $C_4$ (AND MAYBE DIGON?) BOTH GIVING $c_4$ IN $G$
%
%
	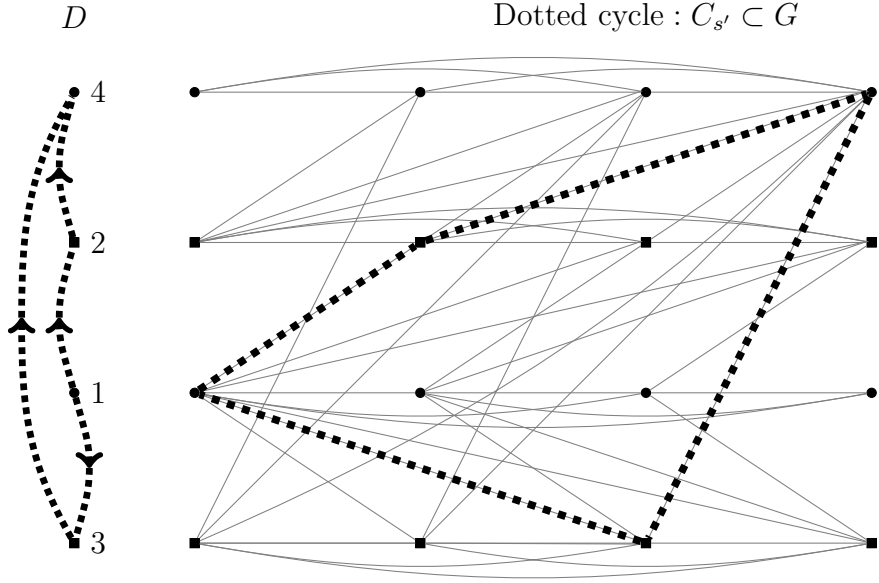
\begin{figure}[htbp]
	\centering
	\begin{tikzpicture}
	[main node1/.style={fill,circle,draw,inner sep=1pt,minimum size=3.5pt}, main node2/.style={fill,rectangle,draw,inner sep=1.25pt,minimum size=3.5pt}, scale=1]

	\begin{scope}[xshift = 0.4cm, yshift=0cm]
	
	    \node[main node1, label=right:4](u3) at (-2,6){};
    \node[main node2, label=right:2](u2) at (-2,4){};
    \node[main node1, label=right:1](u1) at (-2,2){};
    \node[main node2, label=right:3](u0) at (-2,0){};
	
	\node at (-2,7) {$D$};
	    \draw[dotted, line width=0.9mm,->] (u1)to [out=285, in=90] (-1.8,1);
	    \draw [dotted, line width=0.9mm] (-1.8,1) to [out=270,in=75] (u0);
	    \draw[dotted, line width=0.9mm, ->] (u0)to [out=120, in=270](-2.7,3);
	    \draw [dotted, line width=0.9mm] (-2.7,3) to [out=90, in=240] (u3);
	    \draw[dotted, line width=0.9mm,->] (u2)to [out=105, in=270](-2.2,5);
	    \draw[dotted, line width=0.9mm,] (-2.2,5)to [out=90, in=255](u3);
	    \draw[dotted, line width=0.9mm, ->] (u1)to [out=105, in=270](-2.2,3);
	    \draw[dotted, line width=0.9mm] (-2.2,3)to [out=90, in=255](u2);

	\end{scope}
	
		\begin{scope}[xshift = 0cm, yshift=0cm]
	    
%	};
	
    \node[main node1, label=right:](v3) at (0,6){};
    \node[main node2, label=right:](v2) at (0,4){};
    \node[main node1, label=right:](v1) at (0,2){};
        \node[main node2, label=right:](v0) at (0,0){};
%    \node at (0,1){$I_1$};

    \node[main node1, label=below:](w3) at (3,6){};
    \node[main node2, label=below:](w2) at (3,4){};
    \node[main node1, label=above:](w1) at (3,2){};
        \node[main node2, label=above:](w0) at (3,0){};
%    \node at (3,1){$I_2$};
    
    \node[main node1, label=below:](x3) at (6,6){};
    \node[main node2, label=below:](x2) at (6,4){};
    \node[main node1, label=above:](x1) at (6,2){};
        \node[main node2, label=above:](x0) at (6,0){};
%    \node at (6,1){$I_3$};
    
    \node at (6,7){$\text{Dotted cycle}: C_{s'} \subset G$};
    \node[main node1, label=right:](y3) at (9,6){};
    \node[main node2, label=right:](y2) at (9,4){};
    \node[main node1, label=right:](y1) at (9,2){};
       \node[main node2, label=right:](y0) at (9,0){};
%    \node at (9,1){$I_4$};

    \draw[color=gray] (v0)--(w0);
    \draw[color=gray] (v0)--(w3);
    \draw[color=gray] (v3)--(w3)--(v2)--(w2)--(v1)--(w1);
    
    \draw[color=gray] (w3)--(x3)--(w2)--(x2)--(w1)--(x1);
    \draw[color=gray] (x3)--(y3)--(x2)--(y2)--(x1)--(y1);
    \draw[color=gray] (v3) to [out=10, in=170] (y3) -- (v2)to [out=10, in=170](y2)--(v1)to [out=350, in=190](y1);
    \draw[color=gray] (v3) to [out=10,in=170] (x3) --(v2) to [out=10,in=170] (x2) -- (v1) to [out=350,in=190] (x1);
   \draw[color=gray] (w3) to [out=10,in=170] (y3) --(w2) to [out=10,in=170] (y2) -- (w1) to [out=350,in=190] (y1);
    \draw[color=gray] (v0) to [out=350,in=190] (x0) -- (w0) to [out=350,in=190] (y0);
    \draw[color=gray] (v0)to [out=350, in=190](y0);
    \draw[color=gray] (v0)--(w0)--(x0)--(y0);
%    \draw[color=gray] (v0) to [bend right=10] (x2);
    \draw[color=gray] (w0)--(x3)--(v0) to [bend right=10] (y3)--(x0);
    \draw[color=gray] (w0)--(y3); 
    
    \draw[color=gray] (w0)--(v1)--(x0)--(w1)--(y0)--(x1);
    \draw[color=gray] (v1)--(y0);
    %
    %
    % NOW THE COPY OF $C_4$ (DOTTED)
    %
    %
    
   \draw [line width=1mm, dashed] (v1) --(w2)--(y3)--(x0)--(v1);

    \end{scope}
	\end{tikzpicture}
\caption{Illustration of Case 1. For the labelling of $V(D)$ shown, all edges are increasing.\label{fig.DAGcycle}}
\end{figure}

{\em Case 2.} $s' \in [3,s-1]$ and $C$ is the directed cycle $\overrightarrow{C_{s'}}$.

Reversing directions if necessary, assume $C$ is oriented forwards. Then
\[
(1,b_1), (2,b_2), \dots, (s', b_{s'}), (s'+1, b_1), (1, b_1)
\]
is a copy of $C_{s'+1}$ in $G$ (see Figure \ref{fig.twoditriangles}). Here $4\leq s'+1 \leq s$, so this is in the forbidden range.

{\em Case 3.} $s'=3$ and $C$ 
%is not the cyclic triangle $\overrightarrow{C_3}$. That is, $C$
is the transitive triangle $TT_3$. Relabelling if necessary, assume the edges are $b_1 b_2, b_2 b_3$, and $b_1 b_3$.
Then 
\[(1,b_1), (2,b_2), (1,b_2), (2, b_3), (1,b_1)
\]
is a copy of $C_4$ in $G$ (see Figure \ref{fig.twoditriangles}).

%
%
% SECOND PICTURE ILLUSTRATING TRANSITIVE TRIANGLE AND  CYCLIC TRIANGLE BOTH GIVING $C_4$ IN $G$
%
%
	\begin{figure}[htbp]
	\centering
	\begin{tikzpicture}
	[main node1/.style={fill,circle,draw,inner sep=1pt,minimum size=3.5pt}, main node2/.style={fill,rectangle,draw,inner sep=1.25pt,minimum size=3.5pt}, scale=1]
	
%%% 	\draw[fill=gray!10,thick,rounded corners=20] (-0.5,4)--(-0.5,6.5)--(3.5,6.5)--(3.5,-0.5)--(-0.5,-0.5)--(-0.5,4);

	\begin{scope}[xshift = 0.4cm, yshift=0cm]
	
	    \node[main node1, label=right:](u3) at (-2,6){};
    \node[main node2, label=right:](u2) at (-2,4){};
    \node[main node1, label=right:](u1) at (-2,2){};
    \node[main node2, label=right:](u0) at (-2,0){};
	
	\node[left] at (-2.7,3) {$TT_3$};
		\node[below right] at (-1.5,4) {$\overrightarrow{C_{s'}}$};
	
	\node at (-2,7) {$D$};
	    \draw[thick,->] (u0)to [out=120, in=270] (-2.5,2);
	    \draw [thick] (-2.5,2) to [out=90,in=240] (u2);
	    \draw[thick, ->] (u0)to [out=120, in=270](-2.7,3);
	    \draw [thick] (-2.7,3) to [out=90, in=240] (u3);
	    \draw[->, thick] (u2)to [out=105, in=270](-2.2,5);
	    \draw[thick] (-2.2,5)to [out=90, in=255](u3);

	% The cyclic triangle giving the dotted cycle $C_4$:    
	    \draw[dotted, line width=0.9mm, ->] (u1)to [out=105, in=270](-2.2,3);
	    \draw[dotted, line width=0.9mm] (-2.2,3)to [out=90, in=255](u2);
	    \draw[dotted, line width=0.9mm, ->] (u2)to [out=105, in=270](-2.2,5);
	    \draw[dotted, line width=0.9mm] (-2.2,5)to [out=90, in=255](u3);
	    \draw[dotted, ->,line width=0.9mm] (u3)to [out=300, in=90](-1.5,4);
	    \draw[dotted, line width=0.9mm] (-1.5,4)to [out=270, in=60](u1);
	
	\end{scope}
	
		\begin{scope}[xshift = 0cm, yshift=0cm]
	    
%	\node at (4.5,-1) {Example for Lemma \ref{l.smallcycle} when $k=3$
%	};
	
    \node[main node1, label=left:](v3) at (0,6){};
    \node[main node2, label=left:](v2) at (0,4){};
    \node[main node1, label=left:](v1) at (0,2){};
        \node[main node2, label=left:](v0) at (0,0){};
%    \node at (0,1){$I_1$};

%%% \node at (1.5,7){$C_{2k} \subset \Gamma$};

    \node[main node1, label=below:](w3) at (3,6){};
    \node[main node2, label=below:](w2) at (3,4){};
    \node[main node1, label=above:](w1) at (3,2){};
        \node[main node2, label=above:](w0) at (3,0){};
 %%%        \node[above left] at (w0){$v$};
%    \node at (3,1){$I_2$};
    
        \node at (1.5,7){$\text{Solid}: C_{4} \subset G$};
    
    \node[main node1, label=below:](x3) at (6,6){};
    \node[main node2, label=below:](x2) at (6,4){};
    \node[main node1, label=above:](x1) at (6,2){};
        \node[main node2, label=above:](x0) at (6,0){};
%    \node at (6,1){$I_3$};
    
    \node at (6,7){$\text{Dotted}: C_{s'+1} \subset G$};
    \node[main node1, label=right:](y3) at (9,6){};
    \node[main node2, label=right:](y2) at (9,4){};
    \node[main node1, label=right:](y1) at (9,2){};
       \node[main node2, label=right:](y0) at (9,0){};
%    \node at (9,1){$I_4$};
 %%%         \node[above left] at (y0){$v'$};
 
 %%%    \draw[color=gray,ultra thick] (v0)--(w0);
 %%%    \draw[color=gray,ultra thick] (v0)--(w1);
 %%%    \draw[color=gray,ultra thick] (v0)--(w2);
 %%%    \draw[ultra thick] (v1)--(w1)--(v2)--(w2)--(v3)--(w3)--(v1); % these 4 lines are normalised below
    \draw[color=gray] (v0)--(w0);
    \draw[color=gray] (v0)--(w3);
     \draw[color=gray] (v0)--(w2);
    \draw[color=gray] (v3)--(w3)--(v2)--(w2)--(v1);
    \draw [color=gray] (w3)--(v3);
    
    \draw[color=gray] (w3)--(x3)--(w2)--(x2)--(w1)--(x1)--(w3);
    \draw[color=gray] (x3)--(y3)--(x2)--(y2)--(x1)--(y1)--(x3);
    \draw[color=gray] (v3) to [out=10, in=170] (y3) -- (v2)to [out=10, in=170](y2)--(v1)to [out=350, in=190](y1);
    \draw[color=gray] (v3) to [out=10,in=170] (x3) --(v2) to [out=10,in=170] (x2) -- (v1) to [out=350,in=190] (x1) to [out=150,in=315] (v3);
    \draw[color=gray] (w3) to [out=10,in=170] (y3) --(w2) to [out=10,in=170] (y2) -- (w1) to [out=350,in=190] (y1) to [out=150,in=315] (w3);
    \draw [line width=1mm, dashed] (v1) to [out=350,in=190] (y1);
    \draw[line width=1mm, dashed] (v1)--(w2)--(x3)--(y1);
    \draw[color=gray] (v0) to [out=350,in=190] (x0) -- (w0) to [out=350,in=190] (y0);
    \draw[color=gray] (v0)to [out=350, in=190](y0);
    \draw[color=gray] (v0)--(w0)--(x0)--(y0);
    \draw[color=gray] (x3)--(v0)--(y2);
    \draw[color=gray] (v0) to [bend right=10] (x2);
    \draw[color=gray] (v0) to [bend right=10] (y3);
    \draw[color=gray] (x3)--(w0)--(x2);
    \draw[color=gray] (w0)--(y3); 
    \draw[color=gray] (w0) to [bend right=10] (y2);
    \draw[color=gray] (y2)--(x0)--(y3);
    \draw [color=gray] (y1)--(v3)--(w1)--(v1);
    
%% Modified lines to show the $C_4$ arising from the transitive triangle

\draw[ ultra thick] (v0)--(w2)--(v2)--(w3)--(v0);
    
    \end{scope}
	\end{tikzpicture}
\caption{Illustration of Cases 2 (solid) and 3 (dotted). \label{fig.twoditriangles}}
\end{figure}
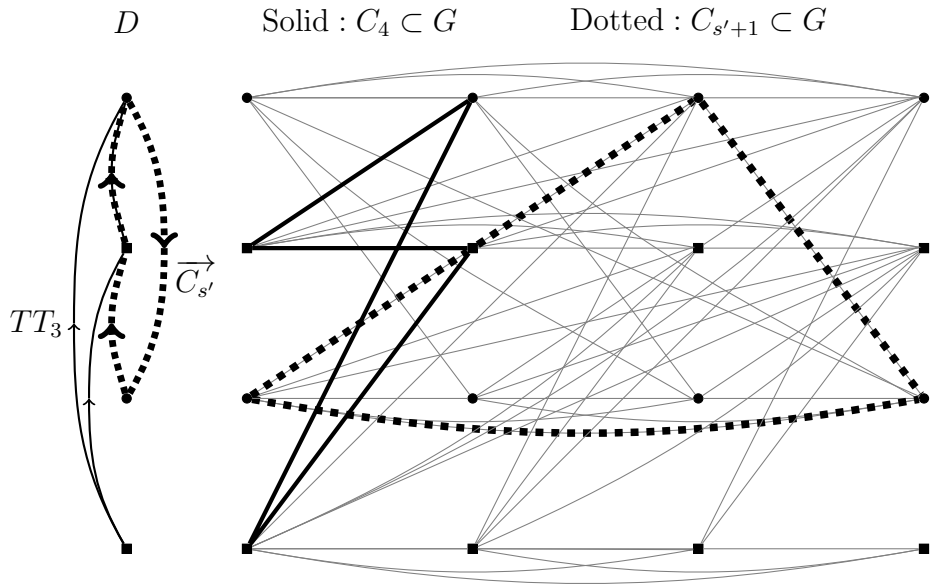

{\em Case 4.} $s'=2$, so $C=\overrightarrow{C_2}$.
Then the following is a copy of $C_4$ in $G$:
\[ (1,b_1), (2,b_1), (1,b_2), (2,b_2), (1,b_1).\]

\underline{a. $\Rightarrow$ b'.:}
Given a digraph $D$ as in a.,
define a strongly repeating graph $G$ on vertex set $\N \times V(D)$ as follows:
%define a bipartite graph $\Gamma$ on vertex set $\{1,2\} \times V(D)$ as follows:
\[
(a,b) \sim (a',b') \text{ in } E(G)
\hspace{5mm}
\Leftrightarrow
\hspace{5mm}
\begin{array}{l}
a < a' \text{ and } bb' \in E(D), \text{ or} \\
a > a' \text{ and } b'b \in E(D), \text{ or} \\
a \neq a' \text{ and } b = b'.
\end{array}
\]

We first claim that $C_4, C_5, \dots, C_s \not <G$. Suppose $(a_1, b_1), \dots, (a_{s'}, b_{s'}), (a_1, b_1)$ is one such cycle for contradiction.

Clearly, $b_1 b_2 \dots b_{s'} b_1$ is a closed walk in $D$, possibly with backward edges and repeated vertices.
%First assume that all $b_1, b_2, \dots, b_s$ are distinct, i.e.\ they in fact form a cycle in $D$. By the assumption on $D$, they must form a \emph{directed} cycle $\Cs$ (in particular $s'=s$), wlog oriented forwards. But then $a_1 \leq a_2 \leq \dots \leq a_s \leq a_1$. So all of the $(a_i,b_i)$ lie in a single column-an independent set-contradicting them forming a cycle.

If any vertex repeats, that is  $b_i=b_j$ for some $i \neq j$,
%We may assume wlog that $i \neq j+1$ (where indices are taken modulo $s'$), as both $i=j+1$ and $j=i+1$ would mean $s'=2$.
%Thus, $j+1\not\in \{i,j\}$.
%Further assume wlog  that the edge $b_j b_{j+1} \in D$ was oriented forwards, that is, it arose from an edge $(a_j,b_j) \sim (a_{j+1},b_{j+1})$ in $G$ with $a_j < a_{j+1}$.
then since $(a_i,b_i) \sim (a_j,b_i)=(a_j,b_j)$ in $G$, it follows that they must be consecutive vertices in the cycle.

So choosing a subsequence of $b_1 b_2 \dots b_{s'}$ by removing one vertex from each  consecutive pair of identical elements, gives an induced cycle in $D$, with at least two vertices, since $s' \geq 4$. 

By the assumption on $D$, $b_1 b_2 \dots b_{s'}$  form a \emph{directed} cycle $\Cs$, so in particular $s'=s$ and no vertices are repeated. Without loss of generality this cycle is oriented forwards. But then $a_1 \leq a_2 \leq \dots \leq a_s \leq a_1$, so equality holds throughout.  This means all all $(a_i,b_i)$ vertices lie in a single column, thus forming an independent set, contradicting the fact that they lie on a cycle.

We claim that
%there is no rainbow independent set of size $m$, with respect to the (infinitely many) independent $n$-sets given by the columns.
this $G$ witnesses $\fr_{\F(C_4, \dots, C_s)}(n,m)=\infty$.

Indeed, suppose $S=\{(a_1, b_1), \dots, (a_m,b_m)\}$ is a rainbow independent $m$-set, with respect to the (infinitely many) columns. The $\{a_i\}$'s are distinct by rainbowness, while the $\{b_i\}$'s are distinct by independence and the strong repeating property.

By assumption, $b_1, \dots, b_m$ contain a cycle, so relabelling as necessary we may assume $b_1b_2 \dots b_{m'}b_1$ is a directed cycle for some $m'\leq m$.
Like before, this tells us $a_1 < a_2 < \dots < a_{m'} < a_1$, a contradiction.
\end{proof}

For general $m$ and $s$, determining the largest $n$ for which 
condition (c) above 
 holds appears difficult. For $s=4$,  standard Ramsey-type results (see \cite{AKS, Shearer}) show that if such  $D$ exists then $n < R(3,m) = O(m^2 / \log m)$.

\section{ $k$-colourable graphs}

Recall that the class of $k$-colourable graphs is denoted by $\cx(k)$.
 
\begin{theorem}\label{t.colourable}
 If  $m \le n$ then   $f_{\cx(k)}(n,m)=k(m-1)+1.$
\end{theorem}
\begin{proof}%[Proof of Theorem~\ref{colourable}]
To show that  $f_{\mathcal{X}(k)}(n,m) > k(m-1)$,
let $G$ be the complete $k$-partite graph with all sides of size $n$, and take a family of $k(m-1)$ independent sets, consisting of  each side of the graph repeated $m-1$ times. A rainbow set of size $m$ must include vertices from two different sides of the graph by the pigeonhole principle, and hence cannot be independent. 
To bound $f_{\mathcal{X}(k)}(n,m)$ from above, let $G$ be a $k$-colourable graph and let $I_1,\ldots,I_{k(m-1)+1} \in \ci_n(G)$.
colour $G$ by colours $V_i$ $(i\le k)$, so $V_1, \ldots ,V_k$ are independent sets covering $V(G)$.

\smallskip

Let $G$ be a $k$-colourable graph and $I_1,\ldots,I_{k(m-1)+1} \in \ci_n(G)$.
Let $M$ be an inclusion-maximal rainbow set.
If $M$ represents all sets $I_j$, then $|M|=k(m-1)+1$, and by pigeonhole $M$ contains $m$ vertices from the same set $V_j$.
Since $V_j$ is independent, this means that $M$ contains a rainbow set of size $m$, as required. 
Thus we may assume that $I_j$ is not represented in $M$ for some $j$. By the maximality of $M$, this implies that  $M \supseteq I_j$, implying in turn that $I_j$ is a rainbow independent set of size $n\ge m$. 
\end{proof}

%Theorem \ref{bdd Delta} can be strengthened, 
Theorem \ref{t.colourable} can be strengthened, in the spirit of Theorem \ref{strongerdrisko}:  a family of independent sets $F_i,~~ i  \le k(m-1)+1$ in a $k$-chromatic graph, where $|F_i|\ge \min(i,n)$, has a rainbow independent $m$-set. To prove this,  follow the same proof as above, choosing the elements of $M$ greedily from the sets $F_1, \ldots ,F_n$.

Moreover, Theorem \ref{t.colourable} can be strengthened in the spirit of Remark \ref{rem.disjoint Kr-free}, namely
\[
\f_\cxk(n,m)=k(m-1)+1 \; (=f_\cxk(n,m)).
\]
The upper bound comes from \eqref{eq.sunflower reduction} and Theorem \ref{t.colourable}, while the lower bound is given by the disjoint union of $m-1$ copies of $K_{\underbrace{n,n,\dots, n}_k}$.

%Let $G$ be $m-1$ disjoint copies of the complete $k$-partite graph with all sides of size $n$,  and take the $k(m-1)$ disjoint independent $n$-sets induced by each $K_{n,n,\dots, n}$.
%Then any set of $m$ vertices has $\ge 2$ vertices $\{x,y\}$ in the same $K_{n,\dots, n}$ by pigeonhole, and hence is either not rainbow (if $x,y$ are in the same class) or not independent (if $x,y$ are in different classes).

\section{Graphs with bounded degrees}\label{bounded_degrees}
Recall that $\cd(k)$ is the class of graphs having  vertex degrees no larger than $k$.
Theorem~\ref{t.colourable}, together with the well-known inequality 
$\chi(G)\leq \Delta(G)+1$,
imply $f_{\cd(k)}(n,n) \leq (k+1)(n-1)+1$.
Applying Brooks' theorem, by which 
$\chi(G)\leq \Delta(G)$ unless  
$G$ is complete or an odd cycle, 
 allows replacing $k+1$ by $k$, for $k \geq 3$ (we omit the details). But probably also the bound $k(n-1)+1$ is   not  sharp.

\begin{conjecture}\label{bdd Delta}
$f_{\cd(k)}(n,n) = \left\lceil \frac{k+1}{2} \right\rceil(n-1)+1$.
\end{conjecture}

One inequality,  $f_{\cd(k)}(n,n) \geq \left\lceil \frac{k+1}{2} \right\rceil(n-1)+1$,  is shown by the graph $G_{t,n}$ from Example~\ref{mainconstruction}. 
As observed there, $G_{t,n}$ has no rainbow independent $n$-set. 
To establish the desired bound, note that every $v \in G_{t,n}$ has degree $2t-2 \leq k$ if we choose $t= \left\lceil \frac{k+1}{2} \right\rceil$.

\begin{theorem}\label{t.degree1or2}
Conjecture \ref{bdd Delta} is true for $k\le 2$. 
\end{theorem}

\begin{proof}
Consider first the case $k\le 1$. % $n$ general. 
%Let $G$ be a graph with $\Delta(G) \le 1$, then $G$ is a union of cliques of size $\leq 2$, and hence $n$ independent sets have a rainbow independent set of size $n$ 
%If $G$ is a matching then $K_3^- \not < G$, so 
Any matching is $K_3^-$-free, so
\[
f_{\cd(1)}(n,m) \leq f_{\cf(K_3^-)}(n,m) = m = \left\lceil \frac{k+1}{2} \right\rceil (m-1)+1,
\]
by \eqref{monotonepartial} and Observation \ref{t.disjointcliques}.

Next consider the case $k=2$, namely of graphs whose components are cycles and paths. We have to show that  in such a graph $G$, any family     $I_1, \ldots, I_{2n-1} \in \ci_n(G)$ has an independent rainbow  $n$-set.
This would follow from Theorem \ref{drisko} if all cycles in $G$ were even, but for general cycles it requires a separate proof. 

Take an inclusion-maximal rainbow set $M$ that induces a bipartite graph in $G$, meaning that it omits at least one vertex from each odd cycle.
If $|M| = 2n-1$, then  $M$ contains a subset of size $n$ in one of the sides of the bipartite graph, which is independent in $G$.
Thus we may assume that $|M| < 2n-1$. Then one of the sets $I_j$, say $I_1$, is not represented by $M$.

Since $M$ is inclusion-maximal rainbow bipartite,  each vertex $v \in I_1 \setminus M$ is contained in an odd cycle $C_v$, say of length $2t(v)+1$, that is a connected component of $G$, such that $V(C_v) \setminus \{v\} \subseteq M$.
For each such $v$  replace in $I_1$
the set $I_1 \cap V(C_v)$ (which is of size at most $t(v)$, since $I_1$ is independent) by any independent set $J_v$ of size $t(v)$ not containing $v$ (note $J_v \subseteq M)$. 
The result of these replacements is an independent $n$-set, contained in $M$, yielding the desired rainbow set.

\end{proof}

\medskip

\subsection{The case $m<n$}
 
 Conjecture \ref{bdd Delta} can be generalized, as follows:

%\begin{theorem}
%$f_{\cd(k)}(n,m) = \left\lceil %\frac{k+1}{n-m+2} \right\rceil(m-1) +1$.
%\end{theorem}

\begin{conjecture}\label{conj m<n}
If 
 $m\le n$ then 
$f_{\cd(k)}(n,m) = \left\lceil \frac{k+1}{n-m+2} \right\rceil(m-1) +1$.
\end{conjecture}

\begin{remark}\label{rem.k=2revisited}
For $k=2,~~m=n-1$ we get $f_{\cd(2)}(n,n-1)=n-1$ which is precisely Conjecture \ref{k=2ab}.
For $m=n$, we get $f_{\cd(2)}(n,n)=2n-1$, which is true by Theorem \ref{t.degree1or2}. The mysterious jump from $f(n,n-1)$ to $f(n,n)$ arises in this formulation in a natural way - it is due to the jump from $\left\lceil \frac{k+1}{n-(n-1)+2} \right\rceil = \left\lceil \frac{3}{3} \right\rceil=1$ to $\left\lceil \frac{k+1}{n-n+2} \right\rceil =\left\lceil \frac{3}{2}\right\rceil =2$.

%here may help to explain the corresponding jump from Conjecture \ref{c.ABversion 1}
\end{remark}

Here is one direction of the conjecture:

\begin{theorem}
$f_{\cd(k)}(n,m) \ge \left\lceil \frac{k+1}{n-m+2} \right\rceil(m-1) +1$.
\end{theorem}

 The examples showing this follow the pattern of Example \ref{mainconstruction}, with some modifications. We describe them explicitly below. Using the terminology of Section 3, they arise naturally as $D$-repeating graphs where the digraph $D$ on vertex set $\mathbb{Z}_n$ contains all arcs from $x$ to $x+i$ for $i=0,1,\dots, r-1$ for some $r$, as in Figure \ref{fig.D(3)construction}.

Specifically, let $r=n-m+2$,
and write $k = r(t-1)+\beta$ for  $\beta\in [0,r)$. Then $t = \frac{k -\beta}{r}+1
= \left\lfloor \frac{k}{r}+1 \right\rfloor
= \left\lceil \frac{k+1}{r} \right\rceil
= \left\lceil \frac{k+1}{n-m+2} \right\rceil$.

Let $G$ be the (repeating) graph on vertex set $[t] \times \mathbb{Z}_n$, with an edge between every pair $(a,b)$ and $(a',b')$ where
\[
a < a' \text{ and }
b'-b \in \{0,1, \dots, r-1\} \text{ (modulo } n).
\]
See Figure~\ref{fig.D(3)construction} for an illustration.
Note that $G$ is $r(t-1)$-regular, and $r(t-1)\leq k$. Certainly, all columns $I_a:= \{a\} \times \mathbb{Z}_n$ are independent $n$-sets.

Claim that every $I \in \ci_m(G)$ is fully contained inside some $I_a$. This way, taking $m-1$ copies of each $I_a$ yields no rainbow independent $m$-set, thus proving $f_{\cd(k)}(n,m) > t(m-1)=\left\lceil \frac{k+1}{n-m+2} \right\rceil(m-1)$.

Take such an $I=\{(a_1, b_1), \dots, (a_m,b_m)\} \in \ci_m(G)$. All rows $[t]\times \{b\}$ are cliques, so the $b_i$'s are distinct. Without loss of generality,
\[
0 \leq b_1 < b_2 < \dots < b_m \leq n-1.
\]
Suppose for contradiction $I$ is not fully contained in any column $\{a\} \times \mathbb{Z}_n$.
Then the $\{a_i\}_{i=1}^m$ are not all equal. So either $a_m<a_1$ or $a_{i-1}< a_{i}$ for some $i$.
In the first case, $b_1-b_m \not \in \{0,1, \dots, r-1\}$ modulo $n$, so $\geq r-n=2-m$, contradicting $b_m \geq b_{m-1}+1\geq \dots \geq b_1 +(m-1)$.
In the latter case, $(a_{i-1}, b_{i-1}) \not\sim (a_{i},b_{i})$ means $b_{i}-b_{i-1} \geq r$, hence
\begin{align*}
n-1 \geq b_m 
%\geq b_{m-1}+1 \geq \dots
&\geq b_{i}+(m-i) \\
&\geq b_{i-1} + (r+m-i) \geq 
%\dots \geq 
b_1 +(r+m-2) \geq n,
\end{align*}
a contradiction.

\begin{figure}
\begin{center}
\begin{tikzpicture}[scale=2.2]
\begin{scope}[rotate=108,xscale=-1]
\node[white, below] at (288:1){123};

		\foreach \x in {0,1,2,...,4} {
		
		\node at (72*\x-18:1.15) {\x};

\draw [ultra thick, ->] (72*\x-18:1) arc [radius=-.15,start angle =72*\x-18, end angle=72*\x+162];
\draw [ultra thick] (72*\x-18:1) arc [radius=.15,start angle =72*\x+162, end angle=72*\x-18];

\begin{scope}[ultra thick,decoration={
    markings,
    mark=at position 0.5 with {\arrow{>}}}
    ] 
    \draw[postaction={decorate}] (72*\x-18:1) to [out=72*\x+119, in=72*\x-9] (72*\x+126:1);
    
    \draw[postaction=decorate] (72*\x-18:1) arc [radius=1,start angle =72*\x-18, end angle=72*\x+54];
\end{scope}

			\draw[fill] (72*\x-18:1) circle [radius=0.05];
		}
		
				\foreach \x in {1} {

\draw [ultra thick, ->] (72*\x-18:1) arc [radius=-.15,start angle =72*\x-18, end angle=72*\x+162];
\draw [ultra thick] (72*\x-18:1) arc [radius=.15,start angle =72*\x+162, end angle=72*\x-18];

\begin{scope}[ultra thick, decoration={
    markings,
    mark=at position 0.5 with {\arrow{>}}}
    ] 
    \draw[postaction={decorate}] (72*\x-18:1) to [out=72*\x+119, in=72*\x-9] (72*\x+126:1);
    
    \draw[postaction=decorate] (72*\x-18:1) arc [radius=1,start angle =72*\x-18, end angle=72*\x+54];
\end{scope}

			\draw[fill] (72*\x-18:1) circle [radius=0.05];
		}

\node at (0,0){$D$};
\end{scope}
\node at (0:1.6){$\Rightarrow$};
	\end{tikzpicture}
	\begin{comment}
\begin{tikzpicture}[scale=2.2]
		\foreach \x in {2,...,5} { \draw [white, ultra thick, ->] (72*\x-18:1) arc [radius=-.15,start angle =72*\x-18, end angle=72*\x+162]; \draw [white, ultra thick] (72*\x-18:1) arc [radius=.15,start angle =72*\x+162, end angle=72*\x-18]; }
\node at (0,0){$G'$};

		\foreach \x in {0,...,5} {
		
			\coordinate (2*\x) at (72*\x:1); 
			\coordinate (2*\x+1) at (72*\x+36:1);

			\draw[fill] (2*\x+1) circle [radius=0.03];
			\draw[fill] (2*\x) circle [radius=0.03];
			}

%\draw (0,0) circle [radius=1];
\draw (2*1+1) arc [start angle = 108, end angle = 288,radius=1];
\draw (2*4) arc [start angle = 288, end angle = 396,radius=1];
\draw [dotted] (2*0+1) arc [start angle=36, end angle = 108, radius =1];			
			\foreach \x in {2,...,5} {
			\draw (2*\x)--(72*\x+108:1);
			}
\draw [dotted] (2*1)--(180:1);

\node[right] at (2*0){123};
\node[right] at (2*0+1){456};
\node[above] at (2*1){123};
\node[above] at (2*1+1){456};
\node[left] at (2*2){123};
\node[left] at (2*2+1){456};
\node[left] at (2*3){123};
\node[below] at (2*3+1){456};
\node[below] at (2*4){123};
\node[right] at (2*4+1){456};

\end{tikzpicture}
\end{comment}
\begin{tikzpicture}[scale=2.2]
\begin{scope} [rotate=108]
		\foreach \x in {2,...,5} { \draw [white, ultra thick, ->] (72*\x-18:1) arc [radius=-.15,start angle =72*\x-18, end angle=72*\x+162]; \draw [white, ultra thick] (72*\x-18:1) arc [radius=.15,start angle =72*\x+162, end angle=72*\x-18]; }
\end{scope}
\node[right] at (0.5,0){$G$};
\draw [rounded corners=3pt] (-1,-1) rectangle (-0.5,1);
\draw [rounded corners=3pt] (-0.95,-0.95) rectangle (-0.55,0.95);
\draw [rounded corners=3pt] (-0.9,-0.9) rectangle (-0.6,0.9);

\draw [rounded corners=3pt] (0,-1) rectangle (0.5,1);
\draw [rounded corners=3pt] (0.05,-0.95) rectangle (0.45,0.95);
\draw [rounded corners=3pt] (0.1,-0.9) rectangle (0.4,0.9);
% add .36 to the y coordinate at a time
% but starting with -0.9 add .18 so half space at top and bottom
\node at (-0.75, -0.72) (10){};
\node at (-0.75, -0.36) (11){};
\node at (-0.75, 0) (12){};
\node at (-0.75, 0.36) (13){};
\node at (-0.75, 0.72) (14){};

\node at (0.25, -0.72) (20){};
\node at (0.25, -0.36) (21){};
\node at (0.25, 0) (22){};
\node at (0.25, 0.36) (23){};
\node at (0.25, 0.72) (24){};

\foreach \x in {0,1,2,3,4}{
\node [left] at (-1,-0.72+0.36*\x){\x};
}

\node [above] at (-0.75,1) {$a$};
\node [above] at (0.25,1) {$a'$};

\foreach \x in {11,12,13,14,10,21,22,23,24,20}{ 
\filldraw (\x) circle [radius =0.03];
}
 \begin{scope} [gray]
\draw (11)--(21);
\draw (11)--(22);
\draw (11)--(23);

\draw (12)--(22);
\draw (12)--(23);
\draw (12)--(24);

\draw (13)--(23);
\draw (13)--(24);
\draw (13)--(20);

\draw (14)--(24);
\draw (14)--(20);
\draw (14)--(21);

\draw (10)--(20);
\draw (10)--(21);
\draw (10)--(22);

\end{scope}

\begin{scope}[ ultra thick, dashed]

\end{scope}

\end{tikzpicture}
\\
\caption{
$f_{\cd(3)}(5,4) > (m-1)t = 6$: 6 independent 5-sets in a $D$-repeating graph $G$ of max degree 3 do not guarantee a rainbow independent 4-set.
\label{fig.D(3)construction}}
\end{center}

\end{figure}
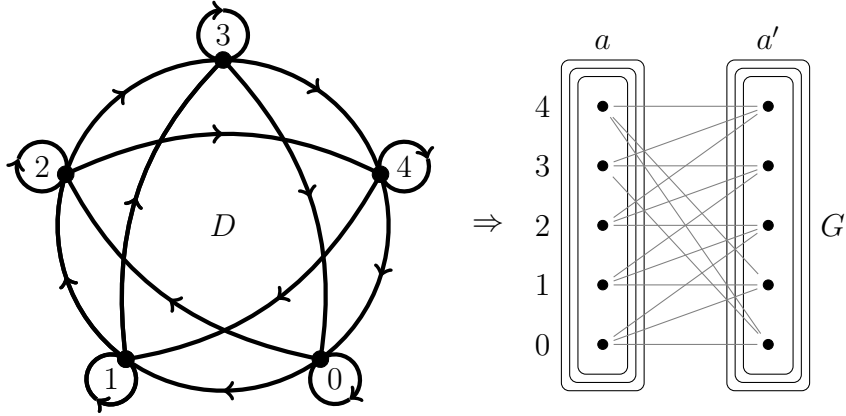

We conclude this section (and the paper) by proving Conjecture \ref{conj m<n} when $m=2$ or $3$, thereby establishing Conjecture \ref{bdd Delta} when $n=2$ or $3$. 

\begin{theorem}\label{n,2}
$f_{\cd(k)}(n,2) = \left\lceil \frac{k+1}{n} \right\rceil +1$.
\end{theorem}
\begin{proof}
Let $G$ be a graph with maximum degree $k$ and $I_1, \ldots , I_{\left\lceil \frac{k+1}{n} \right\rceil +1} \in \ci_n(G)$.
First note $\left\lceil \frac{k+1}{n} \right\rceil +1 \geq 2$.
If $I_i \cap I_j \neq \emptyset$ for some $i \neq j$, then both $I_i$ and $I_j$ contain a rainbow independent $2$-set.
Otherwise the sets $I_1, \ldots , I_{\left\lceil \frac{k+1}{n} \right\rceil +1}$ are mutually disjoint, so $\vert \bigcup_{j > 1} I_j \vert = \left\lceil \frac{k+1}{n} \right\rceil \times n \geq k+1$.

Let $u \in I_1$.
Since the degree of $u$ in $G$ is at most $k$, there must be a vertex $v \in I_j$ for some $j > 1$ which is not adjacent to $u$.
Then $\{u, v\}$ forms a rainbow independent $2$-set in $G$.
\end{proof}

\begin{theorem}\label{thm: 4,3}
$f_{\cd(k)}(n,3) = 2\times \left\lceil \frac{k+1}{n-1} \right\rceil +1$ for $n\geq 3$.
%$f_{\cd(k)}(4,3) = 2\times \left\lceil \frac{k+1}{3} \right\rceil +1$.
\end{theorem}
%The proof is similar to that of Theorem \ref{n=3}, needing just a little more case analysis and 

% Assume wlog $\cup \I = V$.
%We think of $\I$ as the hyperedge set of an $n$-uniform hypergraph on $V$, and write $C$ for the collection of connected components of hyperedges, with a typical component $\C \subset \I$. This way, $|\I|=\sum_{\C\in \C} |\C|$, and $|V| = |\cup \I| =  \sum_{\C \in C}|\cup \C|$.
Henceforth fix $n \geq 3$. We shall need a few auxiliary results.

% @@ PERHAPS THIS CAN BE SHORTENED. PHRASE MAYBE AS ``IN EACH OF THE FOLLOWING, THE SET X IS A RAINBOW INDEPENDENT TRIPLE.''
\begin{observation}\label{obs: three sets}
Let  $I_1, I_2, I_3 \in \ci_n(G)$,
with $I_1 \cap I_2 \neq \emptyset$. 
%Then the following all hold.
Then, in each of the following cases, there is a rainbow independent triple $X$.
%set $X$ is a rainbow independent triple.
\begin{enumerate}[1.]
\item\label{obs 4.5(1)} %If $|I_1 \cap I_2| \geq 2$ and $I_1 \cap I_3 \neq \emptyset$, then there is a subset of $I_1$ which forms a rainbow independent $3$-set.
$\{x,y\} \subset I_1 \cap I_2$ and $x' \in I_1 \cap I_3$ (possibly $x'=x$).
%Then let $X \subset I_1$ be any triple containing all of $x',x,y$.

\item\label{obs 4.5(2)} %If $|I_1 \cap I_2| =1$, say $I_1 \cap I_2 = \{x\}$, and $I_3$ meets $I_1$ at a vertex $y \neq x$, then there is a subset of $I_1$ which forms a rainbow independent $3$-set.
%$|I_1 \cap I_2| =1$, say 
$I_1 \cap I_2 = \{x\}$, and $I_3$ meets $I_1$ at a vertex $y \neq x$.
%Then choose any $X \subset I_1$ containing $\{x,y\}$.

\item\label{obs 4.5(3)}% If $I_1 \cap I_2 = I_2 \cap I_3 = I_1 \cap I_3 = \{u\}$ and there is a vertex $v \in I_i$ and $w \in I_j$ for some $i \neq j$ which are non-adjacent, then $\{u, v, w\}$ forms a rainbow independent set.
$I_1 \cap I_2 = I_2 \cap I_3 = I_1 \cap I_3 = \{u\}$ and there are 2 nonadjacent vertices $v \in I_1$ and $w \in I_2$.
% Then let $X:=\{u,v,w\}$.
 
\item\label{obs 4.5(4)} $I_3$ is disjoint from $I_1 \cup I_2$ and there is some $u \in I_3$, %such that there exist
plus
	\begin{enumerate}
	\item a vertex $v \in I_1 \cap I_2$ which is not adjacent to $u$, and
	\item a vertex $w \neq v$ in $I_1 \cup I_2$ which is not adjacent to $u$.
	\end{enumerate}
%Then choose $X:=\{u, v, w\}$.
\end{enumerate}
\end{observation}
Indeed, in \ref{obs 4.5(1)} (or \ref{obs 4.5(2)}), simply choose any $X \subset I_1$ containing all of $x',x,y$ (resp.\ $x,y$).
In \ref{obs 4.5(3)} and \ref{obs 4.5(4)}, $X:=\{u,v,w\}$ suffices.

This yields structural information regarding the independent sets:

\begin{lemma}\label{lem: ind set comps}
Let $G=(V,E)$ be a graph with a collection $\I$ of independent $n$-sets, 
% with $\cup \; \I = V$,
containing no rainbow independent triple.
Denote by $C_\ci$ the set of connected components of the $n$-uniform hypergraph $(V,\I)$,
and let $\C \in C_\ci$ be one of these components.

\begin{enumerate}
    \item\label{lem 4.12(1)}
    If $\C \supset \{I,J\}$, then every $v \in (\cup \; \ci) \backslash (I \cup J)$ is either adjacent to all of $ I \cap J $, or adjacent to all vertices in $I \cup J$ except for a single vertex in $I \cap J$.
\item\label{lem 4.12(2)}
    If $|\C|=p \geq 3$, then $\C$ forms a sunflower with a core consisting of a single vertex $\{u\}$.
    Moreover, $G[(\cup \C) \backslash \{u\}]$ is a copy of the complete $p$-partite graph $K_{
    %\underbrace{
    n-1,\dots, n-1
    %}_p
    }$, whose parts are the independent sets $I \backslash \{u\}$ for $I \in \C$.
% \item\label{lem 4.12(3)} 

\end{enumerate}
\end{lemma}

\begin{proof}
(1) follows from Observation ~\ref{obs: three sets}\eqref{obs 4.5(4)}.

Next we prove (2). First note Observation ~\ref{obs: three sets}\eqref{obs 4.5(1)} means any distinct $I, J \in \C$ share $\leq 1$ vertex. Suppose in particular that $I \cap J = \{u\}$. Next, 
by Observation ~\ref{obs: three sets}\eqref{obs 4.5(2)}, any other $K \in \C$ meeting $I \cup J$ must do so in precisely $\{u\}$, and by (hypergraph) connectedness it follows $\C$ is a sunflower with core $\{u\}$. Finally, Observation ~\ref{obs: three sets}\eqref{obs 4.5(3)} shows any two non-$u$-vertices from distinct independent sets in $\C$ are adjacent.
\end{proof}

We will also make use of the following:
\begin{lemma}\label{lem.manyverts}
Let $\J$
be a family of independent $n$-sets in a graph with no rainbow independent triple.
Suppose
$\J$ satisfies $|I \cap J| \leq t$ for every distinct $I, J \in \J$.
Then
\[
|\cup \J| \geq 
|\J| \cdot  (n-t/2) +
    \sum_{|\C| \neq 2} \big( |\C| \cdot (t/2-1 \big)+1 \big),
\]
where $C_\J$ is the collection of all components in the $n$-uniform hypergraph whose edge set is $\J$.
In particular,
$| \cup \J|$ is at least:
\begin{enumerate}[(a)]
    \item $(2n-t)( |\J|+1)/2  - % \mathbf{1}_{|\J| \text{ odd}}
    (n-t) $, if $t \geq 2$ and $|\J|$ is odd; or
    \item $(n-1)|\J|$, if $t = 1$; or
    \item $(2n-1)|\J|/2$, if $t \leq 1$ and there is no component in $\J$ of size $\geq 3$.
\end{enumerate}
\end{lemma}

\begin{proof}
%Denote by $C_\J$ the connected components in the $n$-uniform hypergraph whose edge set is $\J$.
For a single component $\C \in C_\J$, 
observe:
\begin{itemize}
    \item if $|\C|=1$, then $|\cup \C|=n=(n-1) |\C|+1$;
    \item if $|\C|=2$, say $\C=\{I,J\}$, then $|\cup \C| = |I| + |J| - |I \cap J| \geq 2n-t =(n-t/2) |\C|$ by definition of $t$; and
    \item if $|\C| \geq 3$, then 
    % $|I \cap J|=1$ for any distinct $I,J \in \C$ by Observation~\ref{obs: three sets}\eqref{obs 4.5(1)}. In fact, Observation~\ref{obs: three sets}\eqref{obs 4.5(2)} shows any $I \in \C$ cannot meet distinct $J,J' \in \C$ in different places, so $\C$ is in fact a sunflower with a size-1 core. Deduce
    Lemma~\ref{lem: ind set comps}\eqref{lem 4.12(2)} gives
    $|\cup \C|=(n-1)|\C|+1$.
\end{itemize}

Putting this all together, 
\begin{align*}
|\cup \J|=  \sum_{\C \in C_\J}|\cup \C|
&= \sum_{|\C|\geq 3}|\cup \C|
    + \sum_{|\C|=2}|\cup \C|
    +  \sum_{|\C| = 1}|\cup \C| \\
%&\geq \sum_{|\C|=2} (2n-t)
%    + \sum_{|\C|\geq 3} \big( (n-1)|\C|+1 \big) 
%    +   \sum_{|\C|=1} n \\
& \geq   \sum_{|\C|=2} (n-t/2)|\C| 
    + \sum_{|\C| \neq 2} \big((n-1)|\C|+1 \big)  \\
&=  \sum_{\C \in C_\J}  (n-t/2)|\C|
    +  \sum_{|\C| \neq 2} \big( |\C| \cdot (t/2-1 \big)+1 \big)  \\
& = |\J| \cdot  (n-t/2) +
    \sum_{|\C| \neq 2} \big( |\C| \cdot (t/2-1 \big)+1 \big).
    %\\
%\text{(if $t \geq 2$) }& \geq 
%|\J| \cdot  (n-t/2)  + |\C_1| \cdot (t/2 - 1) + 1
%\\
%\text{(as $|\C_1|\geq 1$) } & \geq (2n-t)(|\J|+1)/2-n+t.
\end{align*}

For (a), when $\sum |\C| = |\J|
$ is odd, there is at least some $\C_1 \in C_\J$ which is not of size 2. As all other terms in the sum are $\geq 0$ provided $t \geq 2$, we infer
\[
|\cup \J|
\geq
|\J| \cdot  (n-t/2)  + |\C_1| \cdot (t/2 - 1) + 1
\geq (2n-t)(|\J|+1)/2-n+t,
\]
where the latter inequality uses $|\C_1| \geq 1$.
%If $|\J|$ is even, we simply drop the second term in the penultimate line and deduce $|\cup \J| \geq (2n-t)|\J|/2$.

For (b), $t=1$ gives
\[
|\cup \J|
\geq (n-1/2)|\J| - \frac{1}{2}\sum_{|\C| \neq 2}|\C|
\geq (n-1/2)|\J| - \frac{1}{2}|\J|
= (n-1)|\J|.
\]
For (c), all terms in the sum are $\geq 0$ since $|\C|=1$, so deduce $|\cup \J| \geq (2n-t)|\J|/2 \geq (2n-1) |\J|/2$.
\end{proof}

\begin{proof}[Proof of Theorem~\ref{thm: 4,3}]
%    We use induction on $k$ and $n$.

Take $G =(V,E)$ of maximum degree $k$ with a prescribed collection $\I = \{I_1, \dots, I_{2q+1}\}$ of independent $n$-sets, where $q= \left\lceil \frac{k+1}{n-1} \right\rceil$.
	Assume, for contradiction, that $G$ has no rainbow independent triple.

    \smallskip
    
	Let $t := \max\{|I_i \cap I_j|: i \neq j\}$.
    Without loss of generality, we may assume $t=|I_1 \cap I_2|$.
    Write $\ithree$ for $\cup_{j\geq 3} I_j$, and
    let \[A := \{v \in 
%\bigcup_{j \geq 3} I_j
\ithree
    : I_1 \cap I_2 \subseteq N(v) \}
 \;\text{and}\;B:=
 %\bigcup_{j\geq 3} I_j
 \ithree
\setminus A.\]

Note that every $u \in B$ is still adjacent to $t-1$ vertices in $I_1 \cap I_2$ by Lemma \ref{lem: ind set comps}\eqref{lem 4.12(1)}.
	
	\begin{claim}\label{claim t < 2}
	$t \leq 1$.
%	for every $i \neq j$.
	\end{claim}
	\begin{proof}
		Suppose $t \geq 2$. By Lemma~\ref{lem: ind set comps}\eqref{lem 4.12(2)},
		%the union 
		%$\bigcup_{j \geq 3} I_j$
		$\ithree$
		is disjoint from $I_1 \cup I_2$.
 
%    In particular, if $t \geq n-1$, then every vertex in $\bigcup_{j \geq 3} I_j$ has at least $n-1$ neighbors in $I_1 \cup I_2$.
%    Let $G' = G[\bigcup_{j \geq 3}I_j]$.
%    Then we have $\Delta(G') \leq \Delta(G) - (n-1) \leq k - (n-1)$, and hence $q' = \left\lceil \frac{\Delta(G') +1}{n-1} \right\rceil \leq q-1$.
%    By the induction hypothesis, the family $\mathcal{I}' = (I_3, \ldots, I_{2q+1})$ of independent $n$-sets in $G'$ has a rainbow independent set $R$ of size $3$.
%    It is clear that $R$ is a rainbow independent $3$-set in the family $\mathcal{I}$, contradiction.
%    Therefore we may assume that $t < n-1$.
    
    \smallskip
    
    By the above,  double-counting the degree sum of the vertices in $I_1 \cap I_2$ yields:
    %we obtain the inequality 
    \begin{align}\label{eq.doublecount}
   %%  (t-1)|B| \leq
   t|A| + (t-1)|B| = \sum_{u \in I_1\cap I_2} \deg_G(u) \leq tk.    
    \end{align}
    On the other hand, Lemma \ref{lem.manyverts}(a) gives
    \begin{align}\label{4.11 eq2}
    |A|+|B|  = | \ithree  |   
    \geq  %\bigg( 
    (2n-t)q - n + t
    %\bigg)
    \geq
    %\bigg( 
    (2n-t)\frac{k+1}{n-1} - n + t.
    %\bigg).
    \end{align}
 
    Thus,
\begin{align*}
(n-1)|B|
&\geq t\big( (2n-t)(k+1)-(n-t)(n-1)-k(n-1) \big) \\
&=t\big( k(n-t+1)+(2n-t)-(n-t)(n-1) \big) \\
%&=t\big( k(n-t+1)-n(n-3)+t(n-2) \big) \\
(*)&=t\big( k(n-t+1)-(n-t)(n-2) +n \big) \\
&=t\big( (n-t)(k-n+2)+k +n \big).
\end{align*}

\begin{comment}
Suppose first $n=t$, that is, $I_1=I_2$.
Immediately, if any other $I_j$ meets $I_1 = I_2$, then $I_1$ already contains a rainbow independent triple.
Moreover, if any $v \in I_{\geq 3}$ has $<n-1$ edges to $I_1 = I_2$, then it can be extended to a rainbow independent triple using 2 vertices from $I_1=I_2$.
Otherwise, we may form $G'$ from $G$ by removing all vertices in $I_1=I_2$. this  decreases the maximum degree by $\geq n-1$, thus showing $G' \in \cd(k-n+1)$, and moreover does not affect any of the other $2q-1= 2\left\lceil \frac{(k-n+1)+1}{n-1} \right\rceil+1$ independent $n$-sets comprising $\ithree$, which by induction already contain a rainbow independent triple.

Hence, we may assume $n>t$.
This way, $I_1 \setminus I_2 \neq \emptyset$, and any $u \in I_1 \setminus I_2$ is adjacent to all of $B$, thereby showing $|B| \leq d(u) \leq k$.

Applying this to the above then gives

\begin{align*}
    0 & \geq k\big((tn-t^2+1)-(n-1)\big) -t(n-t)(n-2)+nt \\
    & = k\big((n-t)(n-2)\big)
\end{align*}
\end{comment}

%{\color{red}
Since $n \leq |\ithree| = |A| \leq k$ by \eqref{eq.doublecount},  $B \neq \emptyset$. Plus, any $v \in B$ is adjacent to all but 1 vertex of $I_1 \cup I_2$ by Lemma ~\ref{lem: ind set comps}\eqref{lem 4.12(1)}, so it follows
\[
k \geq |I_1 \cup I_2| -1 = 2n-t-1.
\]

Multiplying $(*)$ by $(t-1)/t$, and recalling $|B|(t-1)\leq 
tk$ from \eqref{eq.doublecount}, deduce
\begin{align*}
    0
    &\geq
k \big( (n-t+1)(t-1)-(n-1)\big) -(t-1)(n-t)(n-2)+n(t-1) \\
& = k (n-t)(t-2) -(t-1)(n-t)(n-2)+n(t-1) \\
& = (n-t) \big( k(t-2) -(t-1)(n-2)\big) +n(t-1)\\
& \geq (n-t) \big( (2n-t-1)(t-2) -(t-1)(n-2)\big) +n(t-1)\\
%& = (n-t) ( nt-t^2+3t-3n ) +n(t-1)\\
& = (n-t)^2(t-3) +n(t-1).
\end{align*}

As $t\geq 2$, $n(t-1)>0$. The above then implies $t<3$, namely $t=2$. Moreover, $n-t \neq 0 \Rightarrow n >t$, so $I_1 \setminus I_2 \neq \emptyset$, and any $u \in I_1 \setminus I_2$ is adjacent to all of $B$, thereby strengthening the previous upper bound on $|B|$ from $tk/(t-1)$ to $\Delta(G)=k$.

So we may deduce something stronger from $(*)$ directly:
\begin{align*}
0 & \geq 
k \big( (n-t+1)t-(n-1)\big) -t(n-t)(n-2)+nt \\
& = k\big( 2(n-1)-(n-1)\big)-2(n-2)^2+2n \\
&=(n-1)(k-2n+8) \\
& \geq (n-1) (2n-3-2n+8) = 5(n-1),
\end{align*}
a contradiction.

\end{proof}

%{\color{red}
As in Lemmas \ref{lem: ind set comps} and \ref{lem.manyverts}, let $C_{\ci}$ denote the set of connected components of the $n$-uniform hypergraph $\ci$. 
	\begin{claim}\label{claim nothree}
%	If $|I_{i_1} \cap \cdots \cap I_{i_p}| = 1$ for some $p\geq 2$, then $p = 2$.
    There is no $\C \in C_\ci$ with $\geq 3$ independent sets.
	\end{claim}
	\begin{proof}
%	Without loss of generality, let us assume that $I_1, \ldots, I_p$ is the maximal collection with $|I_1 \cap \cdots \cap I_p| = 1$, say $I_1 \cap \cdots \cap I_p = \{u\}$. 
%	Suppose $p \geq 3$.
%	Then Observation \ref{obs: three sets}\eqref{obs 4.5(2)} ensures that $\{I_1, \dots, I_p\}$ form a sunflower, namely, that the petals $\{I_j \setminus \{u\}\}$ are disjoint sets of size $n-1$.
%	By Observation~\ref{obs: three sets}\eqref{obs 4.5(3)}, we can further assume that $W:=\Big(\bigcup\limits_{j=1}^p I_j\Big)\setminus\{u\}$ induces a complete $p$-partite graph, whose parts are $I_j \setminus \{u\}$, $1 \leq j \leq p$.
Suppose  $|\C|=p \geq 3$, without loss of generality $\C=\{I_{2q+2-p}, \dots, I_{2q+1}\}$.
By Lemma \ref{lem: ind set comps}\eqref{lem 4.12(2)}, there is a vertex $u$ present in all these $I_j$'s such that
%$W:=\Big(\bigcup\limits_{j=1}^p I_j\Big)\setminus\{u\}$
$W:= (\cup \C) \backslash \{u\}$
induces a complete $p$-partite graph whose parts are $I_j \setminus \{u\}$, $2q+2-p \leq j \leq 2q+1$.

%%	By Observation~\ref{obs: three sets}\eqref{obs 4.5(3)} and Theorem~\ref{n,2}, we can further assume that $p \leq q$ and $\left(\bigcup_{1 \leq j \leq p} I_j\right)\setminus\{u\}$ induces a complete $p$-partite graph, whose parts are $I_j \setminus \{u\}$, $1 \leq j \leq p$.
	As before, we write $I_S$ for $\cup_{i\in S}I_i$. Let \[A' = 
%	\bigg( \bigcup_{j = p + 1}^{2q+1} I_j \bigg)
I_{[2q+1-p]}
	\cap  N(u)
 \;\text{and}\; B'=
% \bigg( \bigcup_{j = p + 1}^{2q+1} I_j \bigg)
 I_{[2q+1-p]}
 \setminus N(u).\]
    Clearly $|A'| \leq d(u) \leq k$.
    %since $u$ has at most $k$ neighbors 
    To bound $|B'|$, pick $u' \in I_{2q+1} \setminus\{u\}$. Note $N(u') \supset W \setminus I_{2q+1}$ and $N(u') \supset B'$ by Lemma ~\ref{lem: ind set comps}\eqref{lem 4.12(1)}. So
%    $|B'| \leq k - (p-1)(n-1)$, since any
%    $u' \in I_{2q+1} \setminus \{u\}$
%    has $d(u')\leq k$, despite already being adjacent to all of $B'$ by Lemma ~\ref{lem: ind set comps}\eqref{lem 4.12(1)}, plus all $(p-1)(n-1)$ vertices of $W \setminus I_{2q+1} =\cup_{2q+2-p}^{2q} (I_j \backslash \{u\})$.
\[
k \geq d(u') \geq \sum_{j=2q+2-p}^{2q}|I_j \setminus\{u\}|+|B'|
=(p-1)(n-1)+|B'|.
\]
    Hence we obtain
    \[\vert
%    \bigcup_{j = p+1}^{2q+1} I_j
I_{[2q+1-p]}
    \vert = |A'| + |B'| \leq 2k - (p-1)(n-1).\]
    But this contradicts the lower bound given by Lemma \ref{lem.manyverts}(b):
    \[\vert
%    \bigcup_{j = p+1}^{2q+1} I_j
    I_{[2q+1-p]}
    \vert \geq (n-1)(2q+1-p) \geq 2k+2 - (p-1)(n-1).\]
	\end{proof}

Among all $v \in I_{[2q+1]}$ and sets $S \subset [2q+1]$ of size $q$, choose the pair maximising the quantity
$	\big\vert N(v) \cap I_S \big\vert$ 
%	is maximum among all choices of $v$ and $I_{i_j}$'s.
(recalling $I_S:=\cup_{i \in S}I_i$).
	Without loss of generality, assume $v \in I_{2q+1}$ and $S=[2q] \setminus [q]$, and write this maximum as $\ell:= |N(v) \cap I_{[2q] \setminus [q]}| \leq k$.
	 Lemma \ref{lem.manyverts}(c), together with Claims \ref{claim t < 2} and \ref{claim nothree}, give
	\[|I_{[2q] \setminus [q]}| \geq \left\lceil (2n-2)\frac{q}{2} + \frac{q}{2} \right\rceil
	\geq \left\lceil k+1 + \frac{q}{2} \right\rceil
	\geq k+2
%	\geq d(v)+2
	.\]
	Hence there exists a vertex $v' \in I_{[2q] \setminus [q]} \setminus \{v\}$ not adjacent to $v$. We now focus our efforts on showing that $I_{[q]} \backslash (N(v,v') \cup \{v,v'\}) \neq \emptyset$, where $N(v,v')=N(v) \cup N(v')$. We now put work into  removal of $v$ and $v'$.
	
	Let \[s=|N(v) \cap I_{[2q] \setminus [q]} \cap I_{[q]}|
	%\geq 0
	.\]
	Then, as $N(v) \subset I_{[2q]}=I_{[2q]\setminus[q]} \cup I_{[q]}$, \begin{align*}
	|N(v) \cap I_{[q]}| 
	%&=
%	|N(v)| \\
%	& 
%	= s +|(N(v)  \setminus I_{[2q] \setminus [q]}) \cap I_{[q]}|
%	\\ & 
    = |N(v)| - |N(v) \cap I_{[2q] \setminus [q]}| + s
	\leq
	k - \ell + s.
	\end{align*}
	Write
	$\{u_1, u_2, \dots, u_{s'}\}$ for the elements of $I_{[2q] \setminus [q]} \cap I_{[q]}$ where
	\[s':= |I_{[2q] \setminus [q]} \cap I_{[q]}| \geq s,\]
	then also
%	$s' \geq s$ if $v' \notin I_{[q]}$, and $s' \geq s+1$ otherwise
    $s' \geq s+1$ if $v' \in I_{[q]}$
	(since $v' \in I_{[2q] \setminus [q]} \backslash N(v)$ already).
	
	Each $u_i \in I_{[q]} \cap I_{[2q] \setminus [q]}$ and hence $u_i \in I_{j_i} \cap I_{j_i'}$ for some $j_i \in [q], j'_i \in [2q] \setminus [q]$. As $I_{j_i}$ already meets $I_{j_i'}$, it meets no other $I_j$ by Claim \ref{claim nothree}.

    Moreover, if $j_i=j_{i'}$, then 
    $I_{j_i}=I_{j_{i'}}$ simultaneously meets $I_{j'_{i}}$ and $I_{j'_{i'}}$, hence $I_{j'_i}=I_{j'_{i'}}$. By Claim \ref{claim t < 2}, all intersections between $I_j$'s have size at most one, so
    \[
    \{u_i\} = I_{j_i} \cap I_{j'_i} = I_{j_{i'}} \cap I_{j'_{i'}}
    =\{u_{i'} \},
    \]
	and $i = i'$ follows. Hence these $\{I_{j_i}\}_{i=1}^{s'}$ are all distinct.
%	, as all such intersections $I_{j_i} \cap I_{j'_{i}}$ only have one vertex by Claim \ref{claim t < 2}.

Writing $S'$ for $\{j_1, \dots, j_{s'}\}$, these last two paragraphs show $|I_{S'}|=ns'$, and that $I_{S'}$ is disjoint from $I_{[q] \setminus S'}$. Hence
%    Each of these $s'$ vertices arose from an intersection $I_i \cap I_j$ for some $i \in [q], j \in [2q] \setminus [q]$,    which by Claim~\ref{claim nothree} meet no other such $I_{i'}$ and by Claim ~\ref{claim t < 2} do not share any other such vertices.
 %   @hard to read - is it possible to break into a few sentences?
%		In particular, $I_{1},\ldots, I_{q}$ include $s'$ pairwise disjoint sets, say $\{I_i:i \in S'\}$, that are all disjoint from all $q-1$ others. 
%		Note that: 
%		\begin{equation}\label{eq.observeforlater}
%		    s' \leq q.
%		\end{equation}  
%	Crucially, 	$|I_{S'}|=ns'$ as all its $n$-sets are disjoint, and 	putting $I_{S'}$ aside and applying Lemma~\ref{lem.manyverts}(c) with Claims \ref{claim t < 2} and \ref{claim nothree} 
	\begin{align*}
	    |I_{[q]}|  &= ns' + |I_{[q] \backslash S'}| \\
\text{(by Lemma \ref{lem.manyverts}(c)) }	    &\geq ns' + \left\lceil (2n-1)|[q] \backslash S'|/2\right\rceil \\
	    & = ns' + (n-1)(q-s') +  \left\lceil  \frac{q-s'}{2} \right\rceil \\
%\text{(by \eqref{eq.observeforlater} and choice of $q$) }
        &\geq k + 1 + s' + \left\lceil \frac{q-s'}{2} \right\rceil.
	\end{align*}

	By the maximality property of $\ell$ 
	\[
	|N(v') \cap I_{[q]}| \leq \ell,
	\]
	while 	$|N(v) \cap I_{[q]}| \leq k - \ell + s$ from before. Adding these gives 
	\begin{equation}\label{eq.fewneighboursinIq}
	|I_{[q]} \cap N(v,v')| \leq k+s.
	\end{equation}

	If $s'<q$, then
	\[|I_{[q]} \setminus \{v,v'\}|  \geq \left\lbrace \begin{array}{ll} k+s' & \text{if }v' \in I_{[q]} \\ k+s'+1 & \text{if }v' \notin I_{[q]} \end{array}\right.  > k+s.
	\]
	 	Contrarily, if $s' = q$, then $S'=[q]$, so $I_1, \dots, I_q$ are pairwise disjoint,
%	 	and $v, v' \notin I_{[q]}$.
	 	and each meets exactly one of $I_{q+1}, \dots, I_{2q}$.
	 	%For one, this means $I_{[2q]}$ consists of $q$ pairs of intersecting independent sets, so 
	 	$v \not\in I_{[q]}$ as $I_{2q+1}$ cannot meet $I_{[q]}$ by Claim \ref{claim nothree}. Moreover,
	\[
	|I_{[q]} \setminus \{v,v'\}|=
	\left\lbrace \begin{array}{lllc}
	|I_{[q]}|-1 &= nq-1 &\geq k+q 
	&\text{if }v' \in I_{[q]} \\
	|I_{[q]}| &= nq &\geq k+1+q 
		&\text{if }v' \not\in I_{[q]} \\
	\end{array}
	\right. >k+s.
	\]

In either case, by \eqref{eq.fewneighboursinIq}, there exists a vertex $v'' \in I_{[q]} \setminus \{v, v'\}$ which is adjacent to neither $v$ nor $v'$. Thus $\{v, v', v''\}$ is a rainbow independent triple, as witnessed by $v \in I_{2q+1}, v' \in I_{[2q]\setminus[q]},$ and $v'' \in I_{[q]}$.

%}

	\begin{comment}
	Since we have
	\begin{align*}
	&|N(v) \cap (I_{\left\lceil \frac{k+1}{3} \right\rceil + 2} \cup \cdots \cup I_{2\left\lceil \frac{k+1}{3} \right\rceil + 1})| = k-t, \\
	&|N(v') \cap (I_{\left\lceil \frac{k+1}{3} \right\rceil + 2} \cup \cdots \cup I_{2\left\lceil \frac{k+1}{3} \right\rceil + 1})| \leq t, \text{ and } \\
	& |I_{\left\lceil \frac{k+1}{3} \right\rceil + 2} \cup \cdots \cup I_{2\left\lceil \frac{k+1}{3} \right\rceil + 1}| > k,
	\end{align*}
	there exists a vertex $v'' \in I_{\left\lceil \frac{k+1}{3} \right\rceil + 2} \cup \cdots \cup I_{2\left\lceil \frac{k+1}{3} \right\rceil + 1}$ which is not adjacent to any of $v$ and $v'$.
	Then $\{v, v', v''\}$ forms a rainbow independent $3$-set.
	\end{comment}
\end{proof}

\begin{comment}
    By applying the same method repeatedly, one can see that there is a constant $c > 0$ such that the case $m =3$ of Conjecture~\ref{conj m<n} is true when $k \geq c n^2$.
    Also, if the independent sets are pairwise disjoint, then we need only $2 \times \left\lceil \frac{k+1}{n} \right\rceil +1$ independent $n$-sets to have a rainbow independent $3$-set.
\end{comment}

\begin{comment}
@@ Ending with a not-so-important remark leaves a taste of ``so what''.

\begin{remark}
If the independent $n$-sets are pairwise disjoint, then the above argument shows $\f_{\cd(k)}(n,3) \leq 2 \times \left\lceil \frac{k+1}{n} \right\rceil +1$
%such sets already suffice to guarantee a rainbow independent $3$-set
(as opposed to $2 \times \left\lceil \frac{k+1}{n-1} \right\rceil +1$). It is not clear in general whether this is tight.
\end{remark}
\end{comment}

\section*{Acknowledgements}
%\textbf{Acknowledgements.} 
We are indebted to Alan Lew for the observation that the value of $f_{\cf(K_r)}(n,m)$ is determined by Ramsey numbers.


\bigskip


\medskip




\begin{thebibliography}{99}

\bibitem{AB} R. Aharoni, E. Berger, Rainbow matchings in $r$-partite $r$-graphs, Electron. J. Combin. 16 (2009) \#R119.
\bibitem{ABCHS} R. Aharoni, E. Berger, M. Chudnovsky, D. Howard, P. Seymour, Large rainbow matchings in general graphs, arXiv: 1611.03648v0.

\bibitem{ABK} R. Aharoni, E. Berger, O. Kfir, Acyclic systems of representatives and acyclic colorings of digraphs, J. Graph Theory. 2008 Nov;59(3):177-89.


\bibitem{AHZ} R. Aharoni, R. Holzman, Z. Jiang, Rainbow fractional matchings,
arXiv:1805.09732v1.
\bibitem{AKZ} R. Aharoni, D. Kotlar, R. Ziv, Uniqueness of the extreme cases in theorems of Drisko and Erd\H{o}s-Ginzburg-Ziv, Eur. J. Combin. 67 (2018) 222-229.

\bibitem{AKS} M. Ajtai, J. Koml{\'o}s, E. Szemer{\'e}di,
A note on Ramsey numbers,
J. Combin. Theory, Ser. A 29 (1980) 354-360.

\bibitem{Al} N. Alon, Multicoloured matchings in hypergraphs, Moscow J. Combin. Number Theory 1 (2011) 3-10.
\bibitem{alwz} R. Alweiss, S. Lovett, K. Wu, J. Zhang, Improved bounds for the sunflower lemma, arXiv:1908.08483 (2019).

\bibitem{Ar} A. Aranda, Line Graphs of Triangle-Free Graphs, http://contacts.ucalgary.ca/ \\
info/math/files/info/unitis/publications/1-5916781/LineGraphsK3f.pdf
\bibitem{Ba} I. B\'ar\'any, A generalization of Carath\'eodory's theorem, Discrete Math. 40 (1982) 141-152.
\bibitem{BGS} J. Bar\'at, A. Gy\'arf\'as, G. S\'ark\"ozy, Rainbow matchings in bipartite multigraphs, Period. Math. Hung. 74 (2017) 108-111.
\bibitem{beineke} L. Beineke, Derived graphs and digraphs, Beitr\"ege zur Graphentheorie (1968) 17-33.
\bibitem{ChL} Chartrand, G. and Lesniak, L., Graphs \& Digraphs, Chapman \& Hall/CRC, 3rd edition, 1996.
\bibitem{drisko} A. A. Drisko, Transversals in row-latin rectangles, J. Combin. Theory, Ser. A 84 (1998) 181-195.

\bibitem{EGZ} P. Erd\H{o}s, A. Ginzburg, A. Ziv, Theorem in the additive number theory, Bull. Res. Council Israel 10F (1961) 41-43.

\bibitem{ER} P. Erd\H{o}s, R. Rado, Intersection theorems for systems of sets, J. of the London Math. Society, 1(1) (1960) 85–90.
%\bibitem{hatcher} A. Hatcher, {\em Algebraic Topology}, Cambridge  University Press. 
\bibitem{He} S. Hedetniemi, Graphs of $(0,1)$-matrices, Recent trends in graph theory. Springer, Berlin, Heidelberg, (1971) 157-171.
\bibitem{kostochka} A. Kostochka, A bound of the cardinality of families not containing $\Delta$-systems, The
Mathematics of Paul Erd¨os II, Springer (1997) 229-235.
%\bibitem{Ki} J.H. Kim, The Ramsey number R(3,t) has order of magnitude $\frac{t^2}{\log t}$, Random Structures \& Algorithms 7 (1995) 173–207.
%\bibitem{km} G. Kalai, R. Meshulam, A topological colourful Helly theorem, Adv. Math. 191 (2005) 305-311.
\bibitem{kotlarziv} D. Kotlar, R. Ziv, A matroidal generalization of results of Drisko and Chappell,
arXiv:1407.7321.
\bibitem{Shearer} J. B. Shearer, A note on the independence number of a triangle-free graph, Discrete Math. 46
(1983), 83-87. 
%\bibitem{We} G. Wegner, $d$-Collapsing and nerves of families of convex sets, Arch. Math. (Basel) 26 (1975) 317-321.

\end{thebibliography}
\end{document}